\documentclass[10pt,a4paper]{article}
\usepackage{pifont}
\usepackage{mathrsfs}
\usepackage{indentfirst}
\usepackage{amsmath,amssymb,amsfonts,amsthm,graphics,graphicx}
\usepackage{makeidx}
\usepackage{color}

\newtheorem{theorem}{Theorem}

\newtheorem{proposition}{Proposition}
\newtheorem{lemma}{Lemma}

\newtheorem{corollary}{Corollary}

\newtheorem{observation}{Observation}

\begin{document}
\title{\Large\bf On extremal graphs with at most $\ell$ internally\\
disjoint Steiner trees connecting any $n-1$
vertices\footnote{Supported by NSFC No.11071130, and the ``973"
program.}}
\author{\small Xueliang Li, Yaping Mao
\\
\small Center for Combinatorics and LPMC-TJKLC
\\
\small Nankai University, Tianjin 300071, China
\\
\small lxl@nankai.edu.cn; maoyaping@ymail.com.}
\date{}
\maketitle
\begin{abstract}
The concept of maximum local connectivity $\overline {\kappa}$ of a
graph was introduced by Bollob\'{a}s. One of the problems about it
is to determine the largest number of edges
$f(n;\overline{\kappa}\leq \ell)$ for graphs of order $n$ that have
local connectivity at most $\ell$. We consider a generalization of
the above concept and problem. For $S\subseteq V(G)$ and $|S|\geq
2$, the \emph{generalized local connectivity} $\kappa(S)$ is the
maximum number of internally disjoint trees connecting $S$ in $G$.
The parameter $\overline{\kappa}_k(G)=max\{\kappa(S)|S\subseteq
V(G),|S|=k\}$ is called the \emph{maximum generalized local
connectivity} of $G$. This paper it to consider the problem of
determining the largest number $f(n;\overline{\kappa}_k\leq \ell)$
of edges for graphs of order $n$ that have maximum generalized local
connectivity at most $\ell$. The exact value of
$f(n;\overline{\kappa}_k\leq \ell)$ for $k=n,n-1$ is determined. For
a general $k$, we construct a graph to obtain a sharp lower bound.

\end{abstract}

\section{Introduction}

All graphs considered in this paper are undirected, finite and
simple. We refer to book \cite{Bondy} for graph theoretical notation
and terminology not described here. For any two distinct vertices
$x$ and $y$ in $G$, the \emph{local connectivity} $\kappa_{G}(x,y)$
is the maximum number of internally disjoint paths connecting $x$
and $y$. Then $\kappa(G)=\min\{\kappa_{G}(x,y)|x,y\in V(G),x\neq
y\}$ is defined as \emph{the connectivity} of $G$. In contrast to
this parameter, $\overline{\kappa}(G)=\max\{\kappa_{G}(x,y)|x,y\in
V(G),x\neq y\}$ , introduced by Bollob\'{a}s, is called the
\emph{maximum local connectivity} of $G$. The problem of determining
the smallest number of edges, $h_1(n;\overline{\kappa}\geq r)$,
which guarantees that any graph with $n$ vertices and
$h_1(n;\overline{\kappa}\geq r)$ edges will contain a pair of
vertices joined by $r$ internally disjoint paths was posed by
Erd\"{o}s and Gallai, see \cite{Bartfai} for details. Bollob\'{a}s
\cite{Bollobas1} considered the problem of determining the largest
number of edges, $f(n;\overline{\kappa}\leq \ell)$, for graphs with
$n$ vertices and local connectivity at most $\ell$, that is,
$f(n;\overline{\kappa}\leq \ell)=\max\{e(G)| |V(G)|=n\ and\
\overline{\kappa}(G)\leq \ell\}$. One can see that
$h_1(n;\overline{\kappa}\geq \ell+1)=f(n;\overline{\kappa}\leq
\ell)+1$. Similarly, let $\lambda_{G}(x,y)$ denote the local
edge-connectivity connecting $x$ and $y$ in $G$. Then
$\lambda(G)=\min\{\lambda_{G}(x,y)|x,y\in V(G),x\neq y\}$ and
$\overline{\lambda}(G)=\max\{\lambda_{G}(x,y)|x,y\in V(G),x\neq y\}$
are the edge-connectivity and maximum local edge-connectivity,
respectively. So the edge version of the above problems can be given
similarly. Set $g(n;\overline{\lambda}\leq \ell)=\max\{e(G)|
|V(G)|=n\ and\ \overline{\lambda}(G)\leq \ell\}$. Let
$h_2(n;\overline{\lambda}\geq r)$ denote the smallest number of
edges which guarantees that any graph with $n$ vertices and
$h_2(n;\overline{\kappa}\geq r)$ edges will contain a pair of
vertices joined by $r$ edge-disjoint paths. Similarly,
$h_2(n;\overline{\lambda}\geq \ell+1)= g(n;\overline{\lambda}\leq
\ell)+1$. The problem of determining the precise value of the
parameters $f(n;\overline{\kappa}\leq \ell)$,
$g(n;\overline{\lambda}\leq \ell)$, $h_1(n;\overline{\kappa}\geq
r)$, or $h_2(n;\overline{\kappa}\geq r)$ has obtained wide attention
and many results have been worked out; see \cite{Bollobas1,
Bollobas2, Bollobas3, Leonard1, Leonard2, Leonard3, Mader1, Mader2,
Thomassen}.

In \cite{LLMao}, we generalized the above classical problem. Before
introducing our generalization, we need some basic concepts and
notions. For a graph $G=(V,E)$ and a set $S\subseteq V$ of at least
two vertices, \emph{an $S$-Steiner tree} or \emph{a Steiner tree
connecting $S$} (\emph{a Steiner tree} for short) is a such subgraph
$T(V',E')$ of $G$ that is a tree with $S\subseteq V'$. Two Steiner
trees $T$ and $T'$ connecting $S$ are \emph{internally disjoint} if
$E(T)\cap E(T')=\varnothing$ and $V(T)\cap V(T')=S$. For $S\subseteq
V(G)$ and $|S|\geq 2$, the \emph{generalized local connectivity}
$\kappa(S)$ is the maximum number of internally disjoint trees
connecting $S$ in $G$. Note that when $|S|=2$ a Steiner tree
connecting $S$ is just a path connecting $S$. For an integer $k$
with $2\leq k\leq n$, the \emph{generalized connectivity},
introduced by Chartrand et al. in 1984 \cite{Chartrand}, is defined
as $\kappa_k(G)=min\{\kappa(S)|S\subseteq V(G),|S|=k\}$. For results
on the generalized connectivity, see \cite{LLSun, LL, LLL, LLZ}.
Similar to the classical maximum local connectivity, we \cite{LLMao}
introduced the parameter
$\overline{\kappa}_k(G)=max\{\kappa(S)|S\subseteq V(G),|S|=k\}$,
which is called the \emph{maximum generalized local connectivity} of
$G$. There we considered the problem of determining the largest
number of edges, $f(n;\overline{\kappa}_k\leq \ell)$, for graphs
with $n$ vertices and maximal generalized local connectivity at most
$\ell$, that is, $f(n;\overline{\kappa}_k\leq \ell)=\max\{e(G)|
|V(G)|=n\ and\ \overline{\kappa}_k(G)\leq \ell\}$. We also
considered the smallest number of edges,
$h_1(n;\overline{\kappa}_k\geq r)$, which guarantees that any graph
with $n$ vertices and $h_1(n;\overline{\kappa}_k\geq r)$ edges will
contain a set $S$ of $k$ vertices such that there are $r$ internally
disjoint $S$-trees. It is easy to check that
$h_1(n;\overline{\kappa}_k\geq \ell+1)=f(n;\overline{\kappa}_k\leq
\ell)+1$ for $0\leq \ell \leq n-\lceil k/2\rceil-1$. In
\cite{LLMao}, we determine that $f(n;\overline{\kappa}_3\leq
2)=2n-3$ for $n\geq 3$ and $n\neq 4$, and
$f(n;\overline{\kappa}_3\leq 2)=2n-2$ for $n=4$. Furthermore, we
characterized graphs attaining these values. For general $\ell$, we
constructed graphs to show that $f(n;\overline{\kappa}_3\leq
\ell)\geq \frac{\ell+2}{2}(n-2)+\frac{1}{2}$ for both $n$ and $k$
odd, and $f(n;\overline{\kappa}_3\leq \ell)\geq
\frac{\ell+2}{2}(n-2)+1$ otherwise.

We continue to study the above problems in this paper. The edge
version of these problems are also introduced and investigated. For
$S\subseteq V(G)$ and $|S|\geq 2$, the \emph{generalized local
edge-connectivity} $\lambda(S)$ is the maximum number of
edge-disjoint trees connecting $S$ in $G$. For an integer $k$ with
$2\leq k\leq n$, the \emph{generalized edge-connectivity} \cite{LMS}
is defined as $\lambda_k(G)=min\{\lambda(S)|S\subseteq
V(G),|S|=k\}$. The parameter
$\overline{\lambda}_k(G)=max\{\lambda(S)|S\subseteq V(G),|S|=k\}$ is
called the \emph{maximum generalized local edge-connectivity} of
$G$. Similarly, $g(n;\overline{\lambda}_k\leq \ell)=\max\{e(G)|
|V(G)|=n\ and\ \overline{\lambda}_k(G)\leq \ell\}$, and
$h_2(n;\overline{\lambda}_k\geq r)$ is the smallest number of edges,
$h_2(n;\overline{\lambda}_k\geq r)$, which guarantees that any graph
with $n$ vertices and $h_2(n;\overline{\lambda}_k\geq r)$ edges will
contain a set $S$ of $k$ vertices such that there are $r$
edge-disjoint $S$-trees. Similarly, $h_2(n;\overline{\lambda}_k\geq
\ell+1)= g(n;\overline{\lambda}_k\leq \ell)+1$ for $0\leq \ell \leq
n-\lceil k/2\rceil-1$.

The following result, due to Nash-Williams and Tutte, will be used
later.

\begin{theorem}\label{th1}(Nash-Williams \cite{Nash},Tutte \cite{Tutte})
A multigraph $G$ contains a system of $\ell$ edge-disjoint spanning
trees if and only if
$$
\|G/\mathscr{P}\|\geq \ell(|\mathscr{P}|-1)
$$
holds for every partition $\mathscr{P}$ of $V(G)$, where
$\|G/\mathscr{P}\|$ denotes the number of edges in $G$ between
distinct blocks of $\mathscr{P}$.
\end{theorem}

With the help of Theorem \ref{th1}, this paper obtains the exact
value of $f(n;\overline{\kappa}_k\leq \ell)$ and
$g(n;\overline{\lambda}_k\leq \ell)$ for $k=n,n-1$. The graphs
attaining these values are characterized. It is not easy to solve
these problems for a general $k \ (3\leq k\leq n)$. So we construct
a graph class to give them a sharp lower bound.

To start with, the following two observations are easily seen.

\begin{observation}\label{obs1}
Let $G$ be a connected graph of order $n$. Then

$(1)$ $\kappa_{k}(G)\leq \lambda_{k}(G)$ and
$\overline{\kappa}_{k}(G)\leq \overline{\lambda}_{k}(G)$;

$(2)$ $\kappa_{k}(G)\leq \overline{\kappa}_{k}(G)$ and
$\lambda_{k}(G)\leq \overline{\lambda}_{k}(G)$.
\end{observation}

\begin{observation}\label{obs2}
If $H$ is a spanning subgraph of $G$ of order $n$, then
$\kappa_{k}(H)\leq \kappa_{k}(G)$, $\lambda_{k}(H)\leq
\lambda_{k}(G)$, $\overline{\kappa}_{k}(H)\leq
\overline{\kappa}_{k}(G)$ and $\overline{\lambda}_{k}(H)\leq
\overline{\lambda}_{k}(G)$.
\end{observation}

In \cite{LMS}, we obtained the exact value of $\lambda_k(K_n)$.

\begin{lemma}\cite{LMS}\label{lem1}
Let $n$ and $k$ be two integers such that $3\leq k\leq n$. Then
$$\lambda_k(K_n)=n-\lceil k/2 \rceil$$
\end{lemma}

From Lemma \ref{lem1}, we can derive sharp bounds of
$\overline{\lambda}_k(G)$.

\begin{observation}\label{obs3}
For a connected graph $G$ of order $n$ and $3\leq k\leq n$, $1\leq
\overline{\lambda}_k(G)\leq n-\lceil k/2 \rceil$. Moreover, the
upper and lower bounds are sharp.
\end{observation}
\begin{proof}
From the definitions of $\overline{\lambda}_k(G)$ and $\lambda_k(G)$
and the symmetricity of a complete graph,
$\overline{\lambda}_k(K_n)= \lambda_k(K_n)=n-\lceil
\frac{k}{2}\rceil$. So for a connected graph $G$ of order $n$ it
follows that $\overline{\lambda}_k(G)\leq \overline{\lambda}_k(K_n)=
n-\lceil \frac{k}{2} \rceil$. Since $G$ is connected,
$\overline{\lambda}_k(G)\geq 1$. So $1\leq
\overline{\lambda}_k(G)\leq n-\lceil \frac{k}{2} \rceil$.
\end{proof}

One can easily check that the complete $K_n$ attains the upper bound
and any tree $T$ of order $n$ attains the lower bound. Combining
Observation \ref{obs3} with $(1)$ of Observation \ref{obs1}, the
following observation is immediate.

\begin{observation}\label{obs4}
For a connected graph $G$ of order $n$ and $3\leq k\leq n$, $1\leq
\overline{\kappa}_k(G)\leq n-\lceil k/2 \rceil$. Moreover, the upper
and lower bounds are sharp.
\end{observation}

\section{The case $k=n$}

In this section, we determine the exact value of
$f(n;\overline{\lambda}_k\leq \ell)$ for the case $k=n$. This is
also a preparation for the next section. From Observation
\ref{obs3}, $1\leq \overline{\lambda}_n(G)\leq
\lfloor\frac{n}{2}\rfloor$. In order to make the parameter
$f(n;\overline{\lambda}_k\leq \ell)$ to be meaningful, we assume
that $1\leq \ell \leq \lfloor\frac{n}{2}\rfloor$. Let us focus on
the the case $1\leq \ell\leq \lfloor\frac{n-4}{2}\rfloor$ and begin
with a lemma derived from Theorem \ref{th1}.

\begin{lemma}\label{lem2}
Let $G$ be a connected graph of order $n \ (n\geq 5)$. If $e(G)\geq
{{n-1}\choose{2}}+\ell \ (1\leq \ell\leq
\lfloor\frac{n-4}{2}\rfloor)$ and $\delta(G)\geq \ell+1$, then $G$
contains $\ell+1$ edge-disjoint spanning trees.
\end{lemma}
\begin{proof}
Let $\mathscr{P}=\bigcup_{i=1}^pV_i$ be a partition of $V(G)$ with
$|V_i|=n_i \ (1\leq i\leq p)$, and $\mathcal {E}_p$ be the set of
edges between distinct blocks of $\mathscr{P}$ in $G$.  It suffices
to show $|\mathcal {E}_p|\geq (\ell+1)(p-1)$ so that we can use
Theorem \ref{th1}.

The case $p=1$ is trivial, thus we assume $p\geq 2$. For $p=2$, we
have $\mathscr{P}=V_1\cup V_2$. Set $|V_1|=n_1$. Then $|V_2|=n-n_1$.
If $n_1=1$ or $n_1=n-1$, then $|\mathcal {E}_2|=|E_G[V_1,V_2]|\geq
\ell+1$ since $\delta(G)\geq \ell+1$. Suppose $2\leq n_1 \leq n-2$.
Then $|\mathcal {E}_2|=|E_G[V_1,V_2]|\geq
{{n-1}\choose{2}}+\ell-{{n_1}\choose{2}}-{{n-n_1}\choose{2}}
=-n_1^2+nn_1+\ell-(n-1)$. Since $2\leq n_1 \leq n-2$, one can see
that $|\mathcal {E}_2|$ attains its minimum value when $n_1=2$ or
$n_1=n-2$. Thus $|\mathcal {E}_2|\geq n-3+\ell\geq \ell+1$. So the
conclusion is true for $p=2$ by Theorem \ref{th1}.

Consider the case $p=n$. To have $|\mathcal {E}_{n}|\geq
(\ell+1)(n-1)$, we must have ${{n-1}\choose{2}}+\ell\geq
(\ell+1)(n-1)$, that is, $(n-2\ell-3)(n-2)\geq 2$. Since $\ell\leq
\lfloor\frac{n-4}{2}\rfloor$, this inequality holds. The case
$p=n-1$ can be proved similarly. Since $|\mathcal {E}_{n-1}|\geq
{{n-1}\choose{2}}+\ell-1$, we need the inequality
$\frac{(n-1)(n-2)}{2}+\ell-1\geq (\ell+1)(n-2)$, that is,
$(n-2\ell-3)(n-3)+(n-5)\geq 0$. Since $\ell\leq
\lfloor\frac{n-4}{2}\rfloor$, this inequality holds.

Let us consider the remaining case $p$ for $3\leq p\leq n-2$.
Clearly, $|\mathcal {E}_p|\geq
e(G)-\sum_{i=1}^p{{n_i}\choose{2}}\geq
{{n-1}\choose{2}}+\ell-\sum_{i=1}^p{{n_i}\choose{2}}$. We will show
that ${{n-1}\choose{2}}+\ell-\sum_{i=1}^p{{n_i}\choose{2}}\geq
(\ell+1)(p-1)$, that is, ${{n-1}\choose{2}}+\ell-(\ell+1)(p-1)\geq
\sum_{i=1}^p{{n_i}\choose{2}}$. Actually, we only need to prove that
$\frac{(n-1)(n-2)}{2}-(\ell+1)(p-2)-1\geq
max\{\sum_{i=1}^p{{n_i}\choose{2}}\}$. Since
$f(n_1,n_2,\cdots,n_p)=\sum_{i=1}^p{{n_i}\choose{2}}$ achieves its
maximum value when $n_1=n_2=\cdots=n_{p-1}=1$ and $n_p=n-p+1$, we
need the inequality $\frac{(n-1)(n-2)}{2}-(\ell+1)(p-2)-1\geq
{{1}\choose{2}}(p-1)+{{n-p+1}\choose{2}}$, that is,
$(n-1)(n-2)-2(\ell+1)(p-2)-2\geq (n-p+1)(n-p)$. Thus this inequality
is equivalent to $(p-2)(2n-p-2\ell-3)\geq 2$. Since $1\leq \ell \leq
\lfloor\frac{n-4}{2}\rfloor$ and $3\leq p\leq n-2$, one can see that
the inequality holds. Thus, $|\mathcal {E}_p|\geq (\ell+1)(p-1)$.
From Theorem \ref{th1}, we know that there exist $\ell+1$
edge-disjoint spanning trees, as desired.
\end{proof}

In \cite{LMS}, the graphs with
$\kappa_k(G)=n-\lceil\frac{k}{2}\rceil$ and
$\lambda_k(G)=n-\lceil\frac{k}{2}\rceil$ were characterized,
respectively.

\begin{lemma}\cite{LMS}\label{lem3}
For a connected graph $G$ of order $n$ and $3\leq k\leq n$,
$\kappa_k(G)=n-\lceil\frac{k}{2}\rceil$ or
$\lambda_k(G)=n-\lceil\frac{k}{2}\rceil$ if and only if $G=K_n$ for
$k$ even; $G=K_n\setminus M$ for $k$ odd, where $M$ is an edge set
such that $0\leq |M|\leq \frac{k-1}{2}$.
\end{lemma}

Note that
$\kappa_n(G)=\lambda_n(G)=\overline{\kappa}_n(G)=\overline{\lambda}_n(G)$.
From the above lemma, we can derive the following corollary.

\begin{corollary}\label{cor1}
For a connected graph $G$ of order $n$,
$\kappa_n(G)=\overline{\kappa}_n(G)=\lambda_n(G)=\overline{\lambda}_n(G)=\lfloor\frac{n}{2}\rfloor$
if and only if $G=K_n$ for $n$ even; $G=K_n\setminus M$ for $n$ odd,
where $M$ is an edge set such that $0\leq |M|\leq \frac{n-1}{2}$.
\end{corollary}

Let $\mathcal {G}_n$ be a graph class obtained from a complete graph
$K_{n-1}$ by adding a vertex $v$ and joining $v$ to $\ell$ vertices
of $K_{n-1}$.

\begin{theorem}\label{th2}
Let $G$ be a connected graph of order $n \ (n\geq 6)$. If
$\overline{\lambda}_n(G)\leq \ell \ (1\leq \ell\leq
\lfloor\frac{n}{2}\rfloor)$, then

$$
e(G)\leq \left\{
\begin{array}{ll}
{{n-1}\choose{2}}+\ell, &if~1\leq \ell\leq
\lfloor\frac{n-4}{2}\rfloor;\\
{{n-1}\choose{2}}+n-2,&if~\ell=
\lfloor\frac{n-2}{2}\rfloor~and~$n$~is~even;\\
{{n-1}\choose{2}}+\frac{n-3}{2},&if~\ell=
\lfloor\frac{n-2}{2}\rfloor~and~$n$~is~odd;\\
{{n}\choose{2}},&if~\ell=\lfloor\frac{n}{2}\rfloor.
\end{array}
\right.
$$
with equality if and only if $G\in \mathcal{G}_n$ for $1\leq
\ell\leq \lfloor\frac{n-4}{2}\rfloor$; $G=K_n\setminus e$ where
$e\in E(K_n)$ for $\ell=\lfloor\frac{n-2}{2}\rfloor$ and $n$ even;
$G=K_n\setminus M$ where $M\subseteq E(K_n)$ and $|M|=\frac{n+1}{2}$
for $\ell=\lfloor\frac{n-2}{2}\rfloor$ and $n$ odd; $G=K_n$ for
$\ell=\lfloor\frac{n}{2}\rfloor$.
\end{theorem}
\begin{proof}
For $1\leq \ell\leq \lfloor\frac{n-4}{2}\rfloor$, if $e(G)\geq
{{n-1}\choose{2}}+(\ell+1)$, then $\delta(G)\geq \ell+1$. From Lemma
\ref{lem2}, $\overline{\lambda}_n(G)\geq \ell+1$, which contradicts
to $\overline{\lambda}_n(G)\leq \ell$. So $e(G)\leq
{{n-1}\choose{2}}+\ell$ for $1\leq \ell\leq
\lfloor\frac{n-4}{2}\rfloor$. For $\ell=\lfloor\frac{n-2}{2}\rfloor$
and $n$ even, $e(G)\leq {{n-1}\choose{2}}+n-2$ by Corollary
\ref{cor1}. By the same reason, $e(G)\leq
{{n-1}\choose{2}}+\frac{n-3}{2}$ for
$\ell=\lfloor\frac{n-2}{2}\rfloor$ and $n$ odd. If
$\ell=\lfloor\frac{n}{2}\rfloor$, then for any connected graph $G$
$\overline{\lambda}_k(G)\leq \ell$ by Observation \ref{obs3}. So
$e(G)\leq {{n}\choose{2}}$.

Now we characterize the graphs attaining the upper bounds. Consider
the case $1\leq \ell\leq \lfloor\frac{n-4}{2}\rfloor$. Suppose that
$G$ is a connected graph such that $e(G)={{n-1}\choose{2}}+\ell$.
Clearly, $\delta(G)\geq \ell$. Assume $\delta(G)\geq \ell+1$. Since
$e(G)={{n-1}\choose{2}}+\ell$, $G$ contains $\ell+1$ edge-disjoint
spanning trees by Lemma \ref{lem2}, namely,
$\overline{\lambda}_n(G)\geq \ell+1$, a contradiction. So
$\delta(G)=\ell$, and hence there exists a vertex $v$ such that
$d_G(v)=\ell$. Clearly, $e(G-v)={{n-1}\choose{2}}$. Thus $G-v$ is a
clique of order $n-1$.  Therefore, $G\in \mathcal{G}_n$. For $n$
even and $\ell=\lfloor\frac{n-2}{2}\rfloor$, let
$e(G)={{n-1}\choose{2}}+n-2$. Obviously, $G=K_n\setminus e$, where
$e\in E(K_n)$. For $n$ odd and $\ell=\lfloor\frac{n-2}{2}\rfloor$,
let $e(G)={{n-1}\choose{2}}+\frac{n-3}{2}$. Clearly, $G=K_n\setminus
M$, where $M\subseteq E(K_n)$ and $|M|=\frac{n+1}{2}$. For
$\ell=\lfloor\frac{n}{2}\rfloor$, if $e(G)={{n}\choose{2}}$, then
$G=K_n$.
\end{proof}

\begin{corollary}
For $1\leq \ell\leq \lfloor\frac{n}{2}\rfloor$ and $n\geq 6$,

$$
f(n;\overline{\kappa}_n\leq \ell)=g(n;\overline{\lambda}_n\leq
\ell)=\left\{
\begin{array}{ll}
{{n-1}\choose{2}}+\ell, &if~1\leq \ell\leq
\lfloor\frac{n-4}{2}\rfloor~or~\ell=
\lfloor\frac{n-2}{2}\rfloor~and~$n$~is~odd;\\
{{n-1}\choose{2}}+2\ell,&if~\ell=
\lfloor\frac{n-2}{2}\rfloor~and~$n$~is~even;\\
{{n}\choose{2}},&if~\ell=\lfloor\frac{n}{2}\rfloor.
\end{array}
\right.
$$
\end{corollary}

\section{The case $k=n-1$}

Before giving our main results, we need some preparations. From
Observation \ref{obs4}, we know that $1\leq
\overline{\kappa}_{n-1}(G)\leq \lfloor\frac{n+1}{2}\rfloor$. So we
only need to consider $1\leq \ell\leq \lfloor\frac{n+1}{2}\rfloor$.
In order to determine the exact value of
$f(n;\overline{\kappa}_{n-1}\leq \ell)$ for a general $\ell \ (1\leq
\ell\leq \lfloor\frac{n+1}{2}\rfloor)$, we first focus on the cases
$\ell=\lfloor\frac{n+1}{2}\rfloor$ and
$\lfloor\frac{n-1}{2}\rfloor$. This is also because by
characterizing the graphs with
$\overline{\kappa}_{n-1}(G)=\lfloor\frac{n+1}{2}\rfloor$ and $
\lfloor\frac{n-1}{2}\rfloor$, we can deal with the difficult case
$\ell=\lfloor\frac{n-3}{2}\rfloor$.

\subsection{The subcases $\ell=\lfloor\frac{n+1}{2}\rfloor$ and $\ell=\lfloor\frac{n-1}{2}\rfloor$}

Let us begin this subsection with a useful lemma in \cite{LMS}.

Let $S\subseteq V(G)$ such that $|S|=k$, and $\mathscr{T}$ be a
maximum set of edge-disjoint trees in $G$ connecting $S$. Let
$\mathscr{T}_1$ be the set of trees in $\mathscr{T}$ whose edges
belong to $E(G[S])$, and $\mathscr{T}_2$ be the set of trees
containing at least one edge of $E_G[S,\bar{S}]$. Thus,
$\mathscr{T}=\mathscr{T}_1\cup \mathscr{T}_2$ (Throughout this
paper, $\mathscr{T}$, $\mathscr{T}_1$, $\mathscr{T}_2$ are defined
in this way).

\begin{lemma}\cite{LMS} \label{lem4}
Let $S\subseteq V(G)$, $|S|=k$ and $T$ be a tree connecting $S$. If
$T\in \mathscr{T}_1$, then $T$ uses $k-1$ edges of $E(G[S])\cup
E_G[S,\bar{S}]$; If $T\in \mathscr{T}_2$, then $T$ uses $k$ edges of
$E(G[S])\cup E_G[S,\bar{S}]$.
\end{lemma}

The following results can be derived from Lemma \ref{lem4}.

\begin{lemma}\label{lem5}
Let $G=K_n\setminus M$ be a connected graph of order $n \ (n\geq
4)$, where $M\subseteq E(K_n)$.

$(1)$ If $n$ is odd and $|M|\geq 1$, then
$\overline{\lambda}_{n-1}(G)<\frac{n+1}{2}$;

$(2)$ If $n$ is even and $|M|\geq \frac{n}{2}$, then
$\overline{\lambda}_{n-1}(G)<\frac{n}{2}$.
\end{lemma}
\begin{proof}
$(1)$ For any $S\subseteq V(G)$ such that $|S|=n-1$, obviously,
$|\bar{S}|=1$ and $e\in E(G[S])\cup E_G[S,\bar{S}]$. Let
$|\mathscr{T}_1|=x$ and $|\mathscr{T}|=y$. Then
$|\mathscr{T}_2|=y-x$. Clearly, $|\mathscr{T}_1|\leq
\lfloor\frac{{{n-1}\choose{2}}}{n-2}\rfloor=\frac{n-1}{2}$. Since
$(n-2)|\mathscr{T}_1|+(n-1)|\mathscr{T}_2|\leq |E(G[S])\cup
E_G[S,\bar{S}]|$, it follows that $(n-2)x+(n-1)(y-x)\leq
{{n}\choose{2}}-1$. Then $\lambda(S)=|\mathscr{T}|=y\leq
\frac{x}{n-1}+\frac{n}{2}-\frac{1}{n-1}\leq
\frac{n+1}{2}-\frac{1}{n-1}<\frac{n+1}{2}$. So
$\overline{\lambda}_{n-1}(G)<\frac{n+1}{2}$.

$(2)$ In this case, for any $S\subseteq V(G)$ such that $|S|=n-1$,
we have $|\bar{S}|=1$ and $e\in E(G[S])\cup E_G[S,\bar{S}]$. Let
$|\mathscr{T}_1|=x$ and $|\mathscr{T}|=y$. Then
$|\mathscr{T}_2|=y-x$. Clearly, $|\mathscr{T}_1|\leq
\lfloor\frac{{{n-1}\choose{2}}}{n-2}\rfloor=\lfloor\frac{n-1}{2}\rfloor
=\frac{n-2}{2}$. Since
$(n-2)|\mathscr{T}_1|+(n-1)|\mathscr{T}_2|\leq |E(G[S])\cup
E_G[S,\bar{S}]|$, it follows that $(n-2)x+(n-1)(y-x)\leq
{{n}\choose{2}}-\frac{n}{2}$. Then $\lambda(S)=|\mathscr{T}|=y\leq
\frac{x}{n-1}+\frac{n}{2}-\frac{n}{2(n-1)}\leq
\frac{n}{2}-\frac{1}{n-1}<\frac{n}{2}$. So
$\overline{\lambda}_{n-1}(G)<\frac{n}{2}$.
\end{proof}

With the help of Lemmas \ref{lem3} and \ref{lem5} and Observation
\ref{obs1}, the graphs with
$\overline{\kappa}_{n-1}(G)=\lfloor\frac{n+1}{2}\rfloor$ can be
characterized now.

\begin{proposition}\label{pro1}
For a connected graph $G$ of order $n \ (n\geq 4)$,
$\overline{\kappa}_{n-1}(G)=\lfloor\frac{n+1}{2}\rfloor$ if and only
if $G=K_n$ for $n$ odd; $G=K_n\setminus M$ for $n$ even, where $M$
is an edge set such that $0\leq |M|\leq \frac{n-2}{2}$.
\end{proposition}
\begin{proof}
Consider the case $n$ odd. Suppose that $G$ is a connected graph
such that $\overline{\kappa}_{n-1}(G)=\frac{n+1}{2}$. In fact, the
complete graph $K_n$ is a unique graph attaining this value. Let
$G=K_n\setminus e$ where $e\in E(K_n)$. From $(1)$ of Lemma
\ref{lem5} and Observation \ref{obs1},
$\overline{\kappa}_{n-1}(G)\leq
\overline{\lambda}_{n-1}(G)<\frac{n+1}{2}$. Conversely, if $G=K_n$,
then $\overline{\kappa}_{n-1}(G)\geq \kappa_{n-1}(G)=\frac{n+1}{2}$
by Lemma \ref{lem3}. Combining this with Observation \ref{obs4},
$\overline{\kappa}_{n-1}(G)=\frac{n+1}{2}$.

Now consider the case $n$ even. Assume that $G$ is a connected graph
such that $\overline{\kappa}_{n-1}(G)=\frac{n}{2}$. If
$G=K_n\setminus M$ such that $|M|\geq \frac{n}{2}$, then
$\overline{\kappa}_{n-1}(G)\leq
\overline{\lambda}_{n-1}(G)<\frac{n}{2}$ by $(2)$ of Lemma
\ref{lem5}. So $0\leq |M|\leq \frac{n-2}{2}$. Conversely, if
$G=K_n\setminus M$ such that $0\leq |M|\leq \frac{n-2}{2}$, then
$\overline{\kappa}_{n-1}(G)\geq \kappa_{n-1}(G)=\frac{n}{2}$ by
Lemma \ref{lem3}. From this together with Observation \ref{obs4},
$\overline{\kappa}_{n-1}(G)=\frac{n}{2}$.
\end{proof}

\begin{proposition}\label{pro2}
For a connected graph $G$ of order $n \ (n\geq 4)$,
$\overline{\lambda}_{n-1}(G)=\lfloor\frac{n+1}{2}\rfloor$ if and
only if $G=K_n$ for $n$ odd; $G=K_n\setminus M$ for $n$ even, where
$M$ is an edge set such that $0\leq |M|\leq \frac{n-2}{2}$.
\end{proposition}
\begin{proof}
Assume that $G$ is a connected graph satisfying the conditions of
Proposition \ref{pro2}. From Observation \ref{obs1} and Proposition
\ref{pro1}, it follows that $\overline{\lambda}_{n-1}(G)\geq
\overline{\kappa}_{n-1}(G)=\lfloor\frac{n+1}{2}\rfloor$. Combining
this with Observation \ref{obs3},
$\overline{\lambda}_{n-1}(G)=\lfloor\frac{n+1}{2}\rfloor$.
Conversely, suppose
$\overline{\lambda}_{n-1}(G)=\lfloor\frac{n+1}{2}\rfloor$. For $n$
odd, if $G=K_n\setminus e$ where $e\in E(K_n)$, then
$\overline{\kappa}_{n-1}(G)\leq
\overline{\lambda}_{n-1}(G)<\frac{n+1}{2}$ by $(1)$ of Lemma
\ref{lem5}. So the complete graph $K_n$ is a unique graph attaining
this value. For $n$ even, if $G=K_n\setminus M$ where $M\in E(K_n)$
such that $|M|\geq \frac{n}{2}$, then
$\overline{\lambda}_{n-1}(G)<\lfloor\frac{n+1}{2}\rfloor$ by $(2)$
of Lemma \ref{lem5}. So $0\leq |M|\leq \frac{n-2}{2}$.
\end{proof}

We now focus our attention on the case
$\ell=\lfloor\frac{n-1}{2}\rfloor$. Before characterizing the graphs
with $\overline{\lambda}_{n-1}(G)=\lfloor\frac{n-1}{2}\rfloor$, we
need the following four lemmas. The notion of a second minimal
degree vertex in a graph $G$ will be used in the sequel. If $G$ has
two or more minimum degree vertices, then, choosing one of them as
the first minimum degree vertex, a {\it second minimal degree
vertex} is defined as any one of the rest minimum degree vertices of
$G$. If $G$ has only one minimum degree vertex, then a {\it second
minimal degree vertex} is as its name, defined as any one of
vertices that have the second minimal degree. Note that a second
minimal degree vertex is usually not unique.

\begin{lemma}\label{lem6}
Let $G=K_n\setminus M$ be a connected graph of order $n$, where
$M\subseteq E(K_n)$.

$(1)$ If $n \ (n\geq 10)$ is even and $|M|\geq \frac{3n-4}{2}$, then
$\overline{\lambda}_{n-1}(G)<\frac{n-1}{2}$;

$(2)$ If $n \ (n\geq 10)$ is even, $n+1\leq |M|\leq \frac{3n-6}{2}$
and there is a second minimal degree vertex, say $u_1$, such that
$d_G(u_1)\leq \frac{n-4}{2}$, then
$\overline{\lambda}_{n-1}(G)<\frac{n-2}{2}$;

$(3)$ If $n \ (n\geq 8)$ is odd and $|M|\geq n-1$, then
$\overline{\lambda}_{n-1}(G)<\frac{n-1}{2}$.
\end{lemma}
\begin{proof}
$(1)$ For any $S\subseteq V(G)$ such that $|S|=n-1$, obviously,
$|\bar{S}|=1$ and $e\in E(G[S])\cup E_G[S,\bar{S}]$. Set
$S=V(G)\setminus v$ where $v\in V(G)$. Since $G$ is connected graph,
it follows that $d_G(v)\geq 1$ and hence $d_{K_n[M]}(v)\leq n-2$. So
$|M\cap K_n[S]|\geq \frac{3n-4}{2}-(n-2)=\frac{n}{2}$ and
$|E(G[S])|\leq {{n-1}\choose{2}}-\frac{n}{2}$. Therefore,
$|\mathscr{T}_1|\leq
\frac{{{n-1}\choose{2}}-\frac{n}{2}}{n-2}=\frac{n-2}{2}-\frac{1}{n-2}<\frac{n-2}{2}$,
namely, $|\mathscr{T}_1|\leq \frac{n-4}{2}$. Let $|\mathscr{T}_1|=x$
and $|\mathscr{T}|=y$. Then $|\mathscr{T}_2|=y-x$. Since
$(n-2)|\mathscr{T}_1|+(n-1)|\mathscr{T}_2|\leq |E(G[S])\cup
E_G[S,\bar{S}]|$, it follows that $(n-2)x+(n-1)(y-x)\leq
{{n}\choose{2}}-\frac{3n-4}{2}$. Then
$\lambda(S)=|\mathscr{T}|=y\leq
\frac{x}{n-1}+\frac{n}{2}-\frac{3n-4}{2(n-1)}\leq
\frac{n-2}{2}-\frac{1}{n-1}<\frac{n-2}{2}$. So
$\overline{\lambda}_{n-1}(G)<\frac{n-2}{2}$.

$(2)$ Let $v$ be the vertex such that $d_G(v)=\delta(G)$. Then
$d_G(v)\leq d_G(u_1)\leq \frac{n-4}{2}$. For any $S\subseteq V(G)$
with $|S|=n-1$, at least one of $u_1,v$ belongs to $S$, say $u_1\in
S$. Hence $\overline{\lambda}_{n-1}(G)\leq \lambda(S)\leq
d_G(u_1)\leq \frac{n-4}{2}<\frac{n-2}{2}$.

$(3)$ The proof of $(3)$ is similar to that of $(1)$, and thus
omitted.
\end{proof}

\begin{lemma}\label{lem7}
Let $H$ be a connected graph of order $n-1$.

$(1)$ If $n \ (n\geq 5)$ is odd, $e(H)\geq {{n-2}\choose{2}}$,
$\delta(H)\geq \frac{n-3}{2}$ and any two vertices of degree
$\frac{n-3}{2}$ are nonadjacent, then $H$ contains $\frac{n-3}{2}$
edge-disjoint spanning trees.

$(2)$ If $n \ (n\geq 7)$ is even, $e(H)\geq
{{n-2}\choose{2}}-\frac{n-2}{2}$, $\delta(H)\geq \frac{n-4}{2}$ and
any two vertices of degree $\frac{n-4}{2}$ are nonadjacent, then $H$
contains $\frac{n-4}{2}$ edge-disjoint spanning trees.
\end{lemma}
\begin{proof}
We only give the proof of $(1)$, $(2)$ can be proved similarly. Let
$\mathscr{P}=\bigcup_{i=1}^pV_i$ be a partition of $V(H)$ with
$|V_i|=n_i \ (1\leq i\leq p)$, and $\mathcal {E}_p$ be the set of
edges between distinct blocks of $\mathscr{P}$ in $H$. It suffices
to show $|\mathcal {E}_p|\geq \frac{n-3}{2}(|\mathscr{P}|-1)$ so
that we can use Theorem \ref{th1}.

The case $p=1$ is trivial, thus we assume $p\geq 2$. For $p=2$, we
have $\mathscr{P}=V_1\cup V_2$. Set $|V_1|=n_1$. Then
$|V_2|=n-1-n_1$. If $n_1=1$ or $n_1=n-2$, then $|\mathcal
{E}_2|=|E_G[V_1,V_2]|\geq \frac{n-3}{2}$ since $\delta(H)\geq
\frac{n-3}{2}$. Suppose $2\leq n_1\leq n-3$. Clearly, $|\mathcal
{E}_2|=|E_G[V_1,V_2]|\geq
{{n-2}\choose{2}}-{{n_1}\choose{2}}-{{n-1-n_1}\choose{2}}
=-n_1^2+(n-1)n_1-(n-2)$. Since $2\leq n_1\leq n-3$, one can see that
$|\mathcal {E}_2|$ attains its minimum value when $n_1=2$ or
$n_1=n-3$. Thus $|\mathcal {E}_2|\geq n-4 \geq \frac{n-3}{2}$ since
$n\geq 5$. So the conclusion holds for $p=2$ by Theorem \ref{th1}.

Now consider the remaining case $p$ with $3\leq p\leq n-1$. Since
$|\mathcal {E}_p|\geq e(H)-\sum_{i=1}^p{{n_i}\choose{2}}\geq
{{n-2}\choose{2}}-\sum_{i=1}^p{{n_i}\choose{2}}$, we need to show
that ${{n-2}\choose{2}}-\sum_{i=1}^p{{n_i}\choose{2}}\geq
\frac{n-3}{2}(p-1)$, that is,
${{n-2}\choose{2}}-\frac{n-3}{2}(p-1)\geq
\sum_{i=1}^p{{n_i}\choose{2}}$. Furthermore, we only need to prove
that ${{n-2}\choose{2}}-\frac{n-3}{2}(p-1)\geq
max\{\sum_{i=1}^p{{n_i}\choose{2}}\}$. Since
$f(n_1,n_2,\cdots,n_p)=\sum_{i=1}^p{{n_i}\choose{2}}$ attains its
maximum value when $n_1=n_2=\cdots=n_{p-1}=1$ and $n_p=n-p$, we need
the inequality ${{n-2}\choose{2}}-\frac{n-3}{2}(p-1)\geq
{{1}\choose{2}}(p-1)+{{n-p}\choose{2}}$, that is, $(p-3)(n-p-1)\geq
0$. Since $3\leq p \leq n-1$, one can see that the inequality holds.
Thus, $|\mathcal {E}_p|\geq \frac{n-3}{2}(p-1)$. From Theorem
\ref{th1}, there exist $\frac{n-3}{2}$ edge-disjoint spanning trees.
\end{proof}

The following theorem, due to Dirac, is well-known.

\begin{theorem}\cite{Bondy}(p-485)\label{th3}
Let $G$ be a simple graph of order $n \ (n\geq 3)$ and minimum
degree $\delta$. If $\delta\geq \frac{n}{2}$, then $G$ is
Hamiltonian.
\end{theorem}

\begin{lemma}\label{lem8}
If $n \ (n\geq 8)$ is odd and $G=K_n\setminus M$ such that
$|M|=n-2$, then $\overline{\kappa}_{n-1}(G)\geq \frac{n-1}{2}$.
\end{lemma}
\begin{proof}
Clearly, $e(G)={{n-1}\choose{2}}+1$. Let $v$ be the vertex such that
$d_G(v)=\delta(G)=r$. Choose $S=V(G)\setminus v$. Then $|S|=n-1$. We
distinguish the following cases to show this lemma.

\textbf{Case 1}. $1\leq \delta(G)\leq \frac{n-1}{2}$.

If $\delta(G)=r=1$, then $e(G-v)={{n-1}\choose{2}}$, which implies
that $G-v$ is a clique of order $n-1$. Obviously, $G[S]$ contains
$\frac{n-1}{2}$ trees connecting $S$, namely,
$\overline{\kappa}_{n-1}(G-v)\geq \frac{n-1}{2}$. Therefore,
$\overline{\kappa}_{n-1}(G)\geq \frac{n-1}{2}$.

Suppose $\delta(G)=r\geq 2$. Since $d_G(v)\leq \frac{n-1}{2}$, it
follows that $d_{K_n[M]}(v)\geq n-1-\frac{n-1}{2}=\frac{n-1}{2}$.
Combining this with $|M|=n-2$, $|M\cap E(K_n[S])|\leq
n-2-\frac{n-1}{2}\leq \frac{n-3}{2}$, namely, $G[S]$ is a graph
obtained from a clique of order $n-1$ by deleting at most
$\frac{n-3}{2}$ edges. So $\delta(G[S])\geq
n-2-\frac{n-3}{2}=\frac{n-1}{2}$. Assume that there exists a vertex
in $S$, say $u_1$, such that $d_{G[S]}(u_1)\leq \frac{n+1}{2}$. That
is $d_{G[S]}(u_1)=\frac{n-1}{2}$ or $d_{G[S]}(u_1)=\frac{n+1}{2}$.
Then $d_{G}(u_1)\leq \frac{n+3}{2}$, and hence $d_{K_n[M]}(u_1)\geq
n-1-\frac{n+3}{2}=\frac{n-5}{2}$. We claim that the degree of each
vertex of $S\setminus u_1$ is larger than $\frac{n+3}{2}$ in $G[S]$.
Assume, to the contrary, that there exists a vertex in $S\setminus
u_1$, say $u_2$, such that $d_{G[S]}(u_2)\leq \frac{n+1}{2}$. Then
$d_{G}(u_2)\leq \frac{n+3}{2}$, and hence $d_{K_n[M]}(u_2)\geq
\frac{n-5}{2}$. Therefore, $|M|\geq
d_{K_n[M]}(v)+d_{K_n[M]}(u_1)+d_{K_n[M]}(u_2)\geq
\frac{n-1}{2}+2\cdot \frac{n-5}{2}=\frac{3n-11}{2}>n-2$, a
contradiction. From the above, we conclude that there exists at most
one vertex in $G[S]$ such that its degree is $\frac{n-1}{2}$ or
$\frac{n+1}{2}$. Since $\delta(G[S])\geq \frac{n-1}{2}$, from
Theorem \ref{th3} $G[S]$ is Hamiltonian and hence $G[S]$ contains a
Hamilton cycle $C$. Let $S=\{u_1,u_2,\cdots,u_{n-1}\}$ such that
$vu_i\in E(G) \ (1\leq i\leq r)$. Clearly, $vu_j\in M \ (r+1\leq
j\leq n-1)$. Then the vertices $u_1,u_2,\cdots,u_{r}$ divide the
cycle $C$ into $r$ paths, say $P_1,P_2,\cdots,P_{r}$; see Figure 1
$(a)$. We choose one edge $e_i\in E(P_i)\ (1\leq i\leq r)$ to delete
that satisfies the following conditions:

\ding {182} if there is no vertex of degree $\frac{n-1}{2}$ in
$G[S]$, then $e_i$ is chosen as any edge in $P_i$;

\ding {183} if there exists one vertex $u$ of degree $\frac{n-1}{2}$
in $G[S]$, then $e_i$ is chosen as any edge in $P_i$ that is
incident with $u$.

\begin{figure}[h,t,b,p]
\begin{center}
\scalebox{0.8}[0.8]{\includegraphics{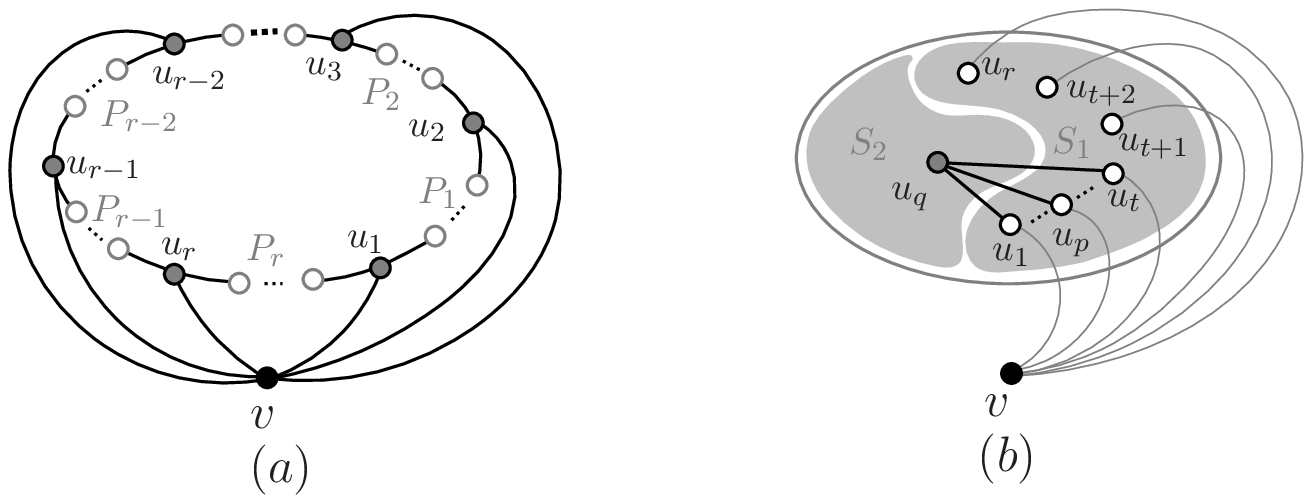}}\\
Figure 1. Graphs for Lemmas \ref{lem8} and \ref{lem9}.
\end{center}
\end{figure}

Then $T=vu_1\cup vu_2\cup\cdots vu_{r}\cup (P_1\setminus e_1)\cup
(P_2\setminus e_2)\cdots (P_{r}\setminus e_r)$ is a Steiner tree
connecting $S$. Set $G_1=G\setminus E(T)$. Clearly,
$\delta(G_1[S])\geq \frac{n-3}{2}$ and there is at most one vertex
of degree $\frac{n-3}{2}$. Combining this with
$e(G_1[S])=e(G)-(n-1)={{n-1}\choose{2}}-(n-2)={{n-2}\choose{2}}$,
$G_1[S]$ contains $\frac{n-3}{2}$ spanning trees by $(1)$ of Lemma
\ref{lem7}. These trees together with the tree $T$ are
$\frac{n-1}{2}$ trees connecting $S$, namely,
$\overline{\kappa}_{n-1}(G)\geq \frac{n-1}{2}$.

\textbf{Case 2}. $\frac{n+1}{2}\leq \delta(G)\leq n-1$.

Let $S=V(G)\setminus v=\{u_1,\cdots,u_{n-1}\}$. Without loss of
generality, let $S_1=\{u_1,\cdots,u_{r}\}$ such that $vu_i\in E(G)$.
Then $\frac{n+1}{2}\leq r\leq n-1$, and $S_2=S\setminus S_1=
\{u_{r+1},\cdots,u_{n-1}\}$. Since $d_G(v)=\delta(G)\geq
\frac{n+1}{2}$, it follows that $|S_1|=r\geq \delta(G)\geq
\frac{n+1}{2}$ and $|S_2|=n-1-r\leq
n-1-\frac{n+1}{2}=\frac{n-3}{2}$. For each $u_j\in S_2 \ (r+1\leq
j\leq n-1)$, $u_j$ has at most $\frac{n-5}{2}$ neighbors in $S_2$
and hence $|E_G[u_j,S_1]|\geq \frac{n+1}{2}-\frac{n-5}{2}=3$ since
$d_G(u_j)\geq \delta(G)\geq \frac{n+1}{2}$. Clearly, the tree
$T'=vu_1\cup vu_2\cup \cdots\cup vu_{r}$ is a Steiner tree
connecting $S_1$. Our idea is to seek for $n-1-r$ edges in
$E_G[S_1,S_2]$ and combine them with $T'$ to form a Steiner tree
connecting $S$. Choose the one with the smallest subscript among the
maximum degree vertices in $S_2$, say $u_{1}'$. Then we search for
the vertex adjacent to $u_1'$ with the smallest subscript among all
the maximum degree vertices in $S_1$, say $u_1''$. Let
$e_1=u_1'u_1''$. Consider the graph $G_1=G\setminus e_1$. Pick up
the one with the smallest subscript among all the maximum degree
vertices in $S_2\setminus u_{1}'$, say $u_{2}'$. Then we search for
the vertex adjacent to $u_2'$ with the smallest subscript among all
the maximum degree vertices in $S_1$, say $u_2''$. Set
$e_2=u_2'u_2''$. We consider the graph $G_2=G_1\setminus
e_2=G\setminus \{e_1,e_2\}$. Choose the one with the smallest
subscript among all the maximum degree vertices in $S_2\setminus
\{u_1',u_2'\}$, say $u_{3}'$. Then we search for the vertex adjacent
to $u_3'$ with the smallest subscript among all the maximum degree
vertices in $S_1$, say $u_3''$. Let $e_3=u_3'u_3''$. We now consider
the graph $G_3=G_2\setminus e_3=G\setminus \{e_1,e_2,e_3\}$. For
each $u_i\in S_2 \ (r+1\leq i\leq n-1)$, we proceed to find
$e_4,e_5,\cdots,e_{n-1-r}$ in the same way. Let
$M'=\{e_1,e_2,\cdots,e_{n-1-r}\}$ and $G_{n-1-r}=G\setminus M'$.
Then $G_{n-1-r}[S]=G[S]\setminus M'$ and the tree $T=vu_{1}\cup
vu_{2}\cup \cdots \cup vu_{r}\cup e_1\cup e_2\cup \cdots \cup
e_{n-1-r}$ is our desired tree. Set $G'=G\setminus E(T)$ (note that
$G'[S]=G_{n-1-r}[S]$).

\noindent {\bf Claim 1}. For each $u_j\in S_1 \ (1\leq j\leq r)$,
$d_{G'[S]}(u_j)\geq \frac{n-1}{2}$.

\noindent{\itshape Proof of Claim $1$.} Assume, to the contrary,
that there exists one vertex $u_p\in S_1$ such that
$d_{G'[S]}(u_p)\leq \frac{n-3}{2}$. By the above procedure, there
exists a vertex $u_q\in S_2$ such that when we pick up the edge
$e_i=u_pu_q$ from $G_{i-1}[S]$ the degree of $u_p$ in $G_{i}[S]$ is
equal to $\frac{n-3}{2}$. That is $d_{G_{i}[S]}(u_p)=\frac{n-3}{2}$
and $d_{G_{i-1}[S]}(u_p)=\frac{n-1}{2}$. From our procedure,
$|E_G[u_q,S_1]|=|E_{G_{i-1}}[u_q,S_1]|$. Without loss of generality,
let $|E_G[u_q,S_1]|=t$ and $u_qu_j\in E(G)$ for $1\leq j\leq t$; see
Figure 1 $(b)$. Thus $u_p\in \{u_1,u_2,\cdots,u_t\}$. Recall that
$|E_G[u_j,S_1]|\geq 3$ for each $u_j\in S_2 \ (r+1\leq j\leq n-1)$.
Since $u_q\in S_2$, we have $t\geq 3$. Clearly, $u_qu_j\notin E(G)$
and hence $u_qu_j\in M$ for $t+1\leq j\leq r$ by our procedure,
namely, $|E_{K_n[M]}[u_q,S_1]|=r-t$. Since
$d_{G_{i-1}[S]}(u_p)=\frac{n-1}{2}$, by our procedure
$d_{G_{i-1}[S]}(u_j)\leq \frac{n-1}{2}$ for each $u_j\in S_1 \
(1\leq j\leq t)$. Assume, to the contrary, that there is a vertex
$u_s \ (1\leq s\leq t)$ such that $d_{G_{i-1}[S]}(u_s)\geq
\frac{n+1}{2}$. Then we should choose the edge $u_qu_s$ instead of
$e_i=u_qu_p$ by our procedure, a contradiction. We conclude that
$d_{G_{i-1}[S]}(u_j)\leq \frac{n-1}{2}$ for each $u_j\in S_1 \
(1\leq i\leq t)$. Clearly, there are at least $n-2-\frac{n-1}{2}$
edges incident to each $u_j \ (1\leq j\leq t)$ that belong to $M\cup
\{e_1,e_2,\cdots,e_{i-1}\}$. Since $i\leq n-1-r$, we have
$\sum_{j=1}^td_{K_n[M]}(u_j)\geq (n-2-\frac{n-1}{2})t-(i-1)>
\frac{n-3}{2}t-(n-1-r)$ and hence $|M|\geq
d_{K_n[M]}(v)+\sum_{j=1}^td_{K_n[M]}(u_j)+|E_{K_n[M]}[u_q,S_1]|>
(n-1-r)+\frac{n-3}{2}t-(n-1-r)+(r-t)=r+\frac{n-5}{2}t\geq
\frac{n+1}{2}+\frac{3(n-5)}{2}=2n-7$, which contradicts to
$|M|=n-2$.

From Claim $1$, $d_{G'[S]}(u_j)\geq \frac{n-1}{2}$ for each $u_j\in
S_1 \ (1\leq i\leq r)$. For each $u_j\in S_2 \ (r+1\leq j\leq n-1)$,
$d_{G'[S]}(u_j)=d_{G[S]}(u_j)-1=d_{G}(u_j)-1\geq \delta(G)-1\geq
\frac{n-1}{2}$. So $\delta(G'[S])\geq \frac{n-1}{2}$. Combining this
with $e(G'[S])=e(G)-(n-1)={{n-2}\choose{2}}$, $G'[S]$ contains
$\frac{n-3}{2}$ spanning trees by From $(1)$ of Lemma \ref{lem7}.
These trees together with the tree $T$ are $\frac{n-1}{2}$ trees
connecting $S$, namely, $\overline{\kappa}_{n-1}(G)\geq
\frac{n-1}{2}$.
\end{proof}

\begin{lemma}\label{lem9}
If $n \ (n\geq 10)$ is even and $G=K_n\setminus M$ such that
$|M|=\frac{3n-6}{2}$ and $d_G(u_1)\geq \frac{n-2}{2}$, then
$\overline{\kappa}_{n-2}(G)\geq \frac{n-2}{2}$, where $u_1$ is a
second minimal degree vertex in $G$.
\end{lemma}
\begin{proof}
It is clear that
$e(G)={{n-2}\choose{2}}+\frac{n}{2}={{n-1}\choose{2}}-\frac{n-4}{2}$.
Let $v$ be the vertex such that $d_G(v)=\delta(G)=r$. Let
$S=V(G)\setminus v=\{u_1,\cdots,u_{n-1}\}$. Without loss of
generality, let $S_1=\{u_1,\cdots,u_{r}\}$ such that $vu_i\in E(G) \
(1\leq i\leq r)$. Then $S_2=S\setminus S_1=
\{u_{r+1},\cdots,u_{n-1}\}$ such that $vu_i\in M \ (r+1\leq i\leq
n-1)$. We have the following two cases to consider.

\textbf{Case 1}. $1\leq \delta(G)\leq \frac{n-2}{2}$.

If $d_G(v)=\delta(G)=1$, then
$e(G-v)={{n-1}\choose{2}}-\frac{n-2}{2}$, which implies that $G-v$
is a graph obtained from a clique of order $n-1$ by deleting
$\frac{n-2}{2}$ edges. From Corollary \ref{cor1} and Observation
\ref{obs1}, $\overline{\kappa}_{n-1}(G-v)=\frac{n-2}{2}$. Therefore,
$\overline{\kappa}_{n-1}(G)\geq \frac{n-2}{2}$. Suppose
$\delta(G)\geq 2$. Since $\delta(G)\leq \frac{n-2}{2}$,
$d_{K_n[M]}(v)\geq n-1-\frac{n-2}{2}=\frac{n}{2}$ and hence $|M\cap
K_n[S]|\leq n-3$. Since $d_{G}(u_1)\geq \frac{n-2}{2}$ where $u_1$
is a second minimal degree vertex, we have $\delta(G[S])\geq
\frac{n-4}{2}$.

First, we consider the case $\delta(G[S])\geq \frac{n}{2}$. We claim
that there are at most two vertices of degree $\frac{n}{2}$ in
$G[S]$. Assume, to the contrary, that there are three vertices of
degree $\frac{n}{2}$ in $G[S]$, say $u_1,u_2,u_3$. Then
$d_{G}(u_i)\leq \frac{n+2}{2}$ for $i=1,2,3$ and hence
$d_{K_n[M]}(u_i)\geq \frac{n-4}{2}$. Therefore, $|M|\geq
d_{K_n[M]}(v)+\sum_{i=1}^3d_{K_n[M]}(u_i)\geq \frac{n}{2}+3\cdot
\frac{n-4}{2}=\frac{4n-12}{2}=2n-6>\frac{3n-6}{2}$, a contradiction.
From the above, we conclude that there exist at most two vertices in
$G[S]$ with degree $\frac{n}{2}$. Since $\delta(G[S])\geq
\frac{n}{2}> \frac{n-1}{2}$, from Theorem \ref{th3} $G[S]$ is
Hamiltonian and hence $G[S]$ contains a Hamilton cycle $C$. Then the
vertices $u_1,u_2,\cdots,u_{r}$ divide the cycle $C$ into $r$ paths,
say $P_1,P_2,\cdots,P_{r}$. We choose one edge $e_i\in E(P_i)\
(1\leq i\leq r)$ to delete that satisfies the following conditions:

\ding {182} if there are two vertices of degree $\frac{n}{2}$, say
$u_1,u_2$ in $G[S]$, then $e_i$ is chosen as any edge in $P_i$ that
is incident with at least one of $u_1,u_2$;

\ding {183} if there is at most one vertex of degree $\frac{n}{2}$,
then $e_i$ is chosen as any edge in $P_i$.

Then $T=vu_1\cup vu_2\cup\cdots vu_{r}\cup (P_1\setminus e_1)\cup
(P_2\setminus e_2)\cdots (P_{r}\setminus e_r)$ is a Steiner tree
connecting $S$. Set $G_1=G\setminus E(T)$. Obviously,
$\delta(G_1[S])\geq \frac{n-4}{2}$ and there is at most one vertex
of degree $\frac{n-4}{2}$. Combining this with
$e(G_1[S])=e(G)-(n-1)={{n-2}\choose{2}}-\frac{n-2}{2}$, $G_1[S]$
contains $\frac{n-4}{2}$ spanning trees by $(2)$ of Lemma
\ref{lem7}. These trees together with the tree $T$ are
$\frac{n-2}{2}$ trees connecting $S$, namely,
$\overline{\kappa}_{n-1}(G)\geq \frac{n-2}{2}$.

Next, we focus on the case that $\delta(G[S])=\frac{n-2}{2}$ and
$\delta(G[S])=\frac{n-4}{2}$. If $\delta(G[S])=\frac{n-4}{2}$, then
there exists a vertex, say $u_1$, such that $d_{G[S]}(u_1)=
\frac{n-4}{2}$. Since the degree of a second minimal degree vertex
is not less than $\frac{n-2}{2}$, we have $u_1\in S_1$. Thus
$d_{G}(u_1)=\frac{n-2}{2}$ and $vu_1\in E(G)$. If
$\delta(G[S])=\frac{n-2}{2}$, then there exists a vertex, say $u_1$,
such that $d_{G[S]}(u_1)=\frac{n-2}{2}$ and $u_1\in S_1$, or
$d_{G[S]}(u_1)=\frac{n-2}{2}$ and $u_1\in S_2$. Thus
$d_{G}(u_1)=\frac{n}{2}$ and $u_1\in S_1$, or
$d_{G}(u_1)=\frac{n-2}{2}$ and $u_1\in S_2$. We only give the proof
of the case that $d_{G}(u_1)=\frac{n}{2}$ and $u_1\in S_1$. The
other two cases can be proved similarly.

Suppose $d_{G}(u_1)=\frac{n}{2}$ and $u_1\in S_1$. Similar to the
proof of Lemma \ref{lem8}, we want to find out a tree connecting $S$
with root $v$, say $T$. Let $G_1=G\setminus E(T)$. We hope that the
graph $G_1[S]$ satisfies the conditions of $(2)$ of Lemma
\ref{lem7}. Thus there are $\frac{n-4}{2}$ spanning trees connecting
$S$ in $G_1[S]$. These trees together with the tree $T$ are
$\frac{n-2}{2}$ trees connecting $S$, namely,
$\overline{\kappa}_{n-1}(G)\geq \frac{n-2}{2}$. Let
$S_1'=S_1\setminus u_1$ and $S'=S_1'\cup S_2$. Let us focus on the
graph $G[S_1']$. If $r=2$, then $G[S']$ is a graph obtained from a
clique of order $n-2$ by deleting one edge since
$d_{K_n[M]}(u_1)=\frac{n-2}{2}$ and $d_{K_n[M]}(v)=n-3$ and
$|M|=\frac{3n-6}{2}$. Without loss of generality, let
$N_G(v)=\{u_1,u_2\}$. Clearly, $G[S']$ contains a Hamilton path $P$
with $u_2$ as one of its endpoints. Then $T=vu_1\cup vu_2\cup P$.
Set $G_1=G\setminus E(T)$. Thus $\delta(G_1[S'])=\delta(G[S'])-2\geq
n-4-2=n-6\geq \frac{n-2}{2}$. Combining this with
$d_{G_1[S]}(u_1)=\frac{n-2}{2}$, the result follows by Lemma
\ref{lem7}. Now assume $r\geq 3$. Since
$d_{K_n[M]}(u_1)=\frac{n-2}{2}$, $d_{K_n[M]}(v)\geq \frac{n}{2}$ and
$|M|=\frac{3n-6}{2}$, $G[S']$ is a graph obtained from the complete
graph $K_{n-2}$ by deleting at most $\frac{n-4}{2}$ edges and hence
$\delta(G[S'])\geq n-3-\frac{n-4}{2}=\frac{n-2}{2}$. It is clear
that there exist at least two vertices of degree $n-3$ and there is
also at most one vertex of degree $\frac{n-2}{2}$ in $G[S']$.
Without loss of generality, let $u_{i_1},u_{i_2}$ be two vertices of
degree $n-3$.

\begin{figure}[h,t,b,p]
\begin{center}
\scalebox{0.8}[0.8]{\includegraphics{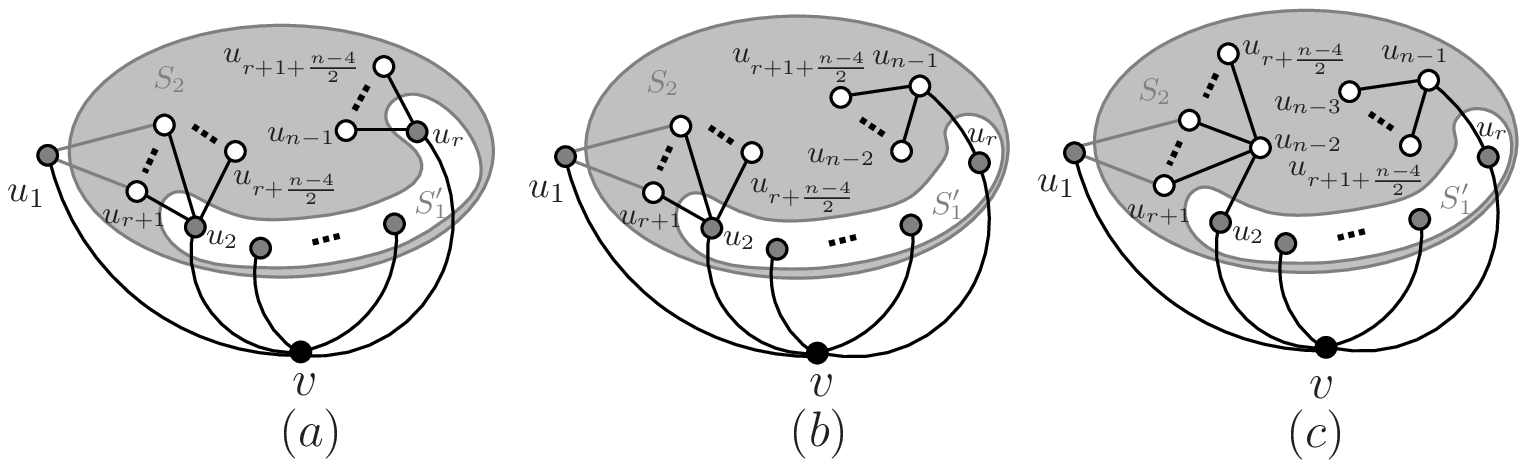}}\\
Figure 2. Graphs for Case $1$ of Lemma \ref{lem9}.
\end{center}
\end{figure}

If $u_{i_1},u_{i_2}\in S_1'$, without loss of generality, let
$u_{i_1}=u_2$ and $u_{i_2}=u_{r}$, then the tree $T=vu_1\cup
\cdots\cup vu_{r}\cup u_2u_{r+1}\cup \cdots\cup
u_2u_{r+\frac{n-4}{2}}\cup u_{r}u_{r+\frac{n-4}{2}+1}\cup \cdots\cup
u_{r}u_{n-1}$ is a Steiner tree connecting $S$; see Figure 2 $(a)$.
Set $G_1=G\setminus E(T)$. Clearly, $d_{G_1[S]}(u_1)=
\frac{n-2}{2}$, $d_{G_1[S]}(u_2)\geq n-3-\frac{n-4}{2}\geq
\frac{n-2}{2}$ and
$d_{G_1[S]}(u_{r})=(n-3)-(n-1-r-\frac{n-4}{2})=r-2+\frac{n-4}{2}\geq
\frac{n-2}{2}$. For $u_i\in S_2$ ($r+1\leq i\leq n-1$),
$d_{G_1[S]}(u_{i})\geq \frac{n-4}{2}$ and there is at most one
vertex of degree $\frac{n-4}{2}$ in $G_1[S]$. So $\delta(G_1[S])\geq
\frac{n-4}{2}$ and there is at most one vertex of degree
$\frac{n-4}{2}$ in $G_1[S]$, as desired. If $u_{i_1}\in S_1'$ and
$u_{i_2}\in S_2$, without loss of generality, let $u_{i_1}=u_2$ and
$u_{i_2}=u_{n-1}$, then the tree $T=vu_1\cup \cdots\cup vu_{r}\cup
u_2u_{r+1}\cup \cdots\cup u_2u_{r+\frac{n-4}{2}}\cup
u_{n-1}u_{r+\frac{n-4}{2}+1}\cup \cdots\cup u_{n-1}u_{n-2}\cup
u_{n-1}u_{r}$ is our desired tree; see Figure 2 $(b)$. Set
$G_1=G\setminus E(T)$. One can see that $\delta(G_1[S])\geq
\frac{n-4}{2}$ and there is at most one vertex of degree
$\frac{n-4}{2}$ in $G_1[S]$, as desired. Let us consider the
remaining case $u_{i_1},u_{i_2}\in S_2$. Without loss of generality,
let $u_{i_1}=u_{n-1}$ and $u_{i_2}=u_{n-2}$. The tree $T=vu_1\cup
\cdots\cup vu_{r}\cup u_{n-2}u_{r+1}\cup \cdots\cup
u_{n-2}u_{r+\frac{n-4}{2}}\cup u_{n-1}u_{r+\frac{n-4}{2}+1}\cup
\cdots\cup u_{n-1}u_{n-3}\cup u_{2}u_{n-2}\cup u_{n-1}u_{r}$ is our
desired tree; see Figure 2 $(c)$. Set $G_1=G\setminus E(T)$. One can
see that $\delta(G_1[S])\geq \frac{n-4}{2}$ and there is at most one
vertex of degree $\frac{n-4}{2}$ in $G_1[S]$. Using $(2)$ of Lemma
\ref{lem7}, we can get $\overline{\kappa}_{n-1}(G)\geq
\frac{n-2}{2}$.

\textbf{Case 2}. $\frac{n}{2}\leq \delta(G)\leq n-1$.

Recall that $S_1=\{u_1,\cdots,u_{r}\}$ with $vu_i\in E(G)$ and
$S_2=S\setminus S_1= \{u_{r+1},\cdots,u_{n-1}\}$. Obviously,
$|S_1|=r=\delta(G)\geq \frac{n}{2}$ and $|S_2|=n-1-r\leq
n-1-\frac{n}{2}=\frac{n-2}{2}$. For each $u_j\in S_2 \ (r+1\leq
j\leq n-1)$, $u_j$ has at most $\frac{n-4}{2}$ neighbors in $S_2$
and hence $|E_G[u_j,S_1]|\geq \frac{n}{2}-\frac{n-4}{2}=2$ since
$d_G(u_j)\geq \delta(G)\geq \frac{n}{2}$. Clearly, the tree
$T'=vu_1\cup vu_2\cup \cdots\cup vu_{r}$ is a Steiner tree
connecting $S_1$. Our idea is to seek for $n-1-r$ edges in
$E_G[S_1,S_2]$ and combine them with $T'$ to form a Steiner tree
connecting $S$. We employ the method used in Case $2$ of Lemma
\ref{lem8}. Choose the one with the smallest subscript among all the
maximum degree vertices in $S_2$, say $u_{1}'$. Then we search for
the vertex adjacent to $u_1'$ with the smallest subscript among all
the maximum degree vertices in $S_1$, say $u_1''$. Let
$e_1=u_1'u_1''$. Consider the graph $G_1=G\setminus e_1$. Pick up
the one with the smallest subscript among all the maximum degree
vertices in $S_2\setminus u_{1}'$, say $u_{2}'$. Then we search for
the vertex adjacent to $u_2'$ with the smallest subscript among all
the maximum degree vertices in $S_1$, say $u_2''$. Set
$e_2=u_2'u_2''$. We consider the graph $G_2=G_1\setminus
e_1=G\setminus \{e_1,e_2\}$. For each $u_i\in S_2 \ (r+1\leq i\leq
n-1)$, we proceed to find $e_3,e_4,\cdots,e_{n-1-r}$ in the same
way. Let $M'=\{e_1,e_2,\cdots,e_{n-1-r}\}$ and $G_{n-1-r}=G\setminus
M'$. Then $G_{n-1-r}[S]=G[S]\setminus M'$ and the tree $T=vu_{1}\cup
vu_{2}\cup \cdots \cup vu_{r}\cup e_1\cup e_2\cup \cdots \cup
e_{n-1-r}$ is our desired tree. Set $G'=G\setminus E(T)$ (note that
$G'[S]=G_{n-1-r}[S]$).

\noindent {\bf Claim 2}. For each $u_j\in S_1 \ (1\leq j\leq r)$,
$d_{G'[S]}(u_j)\geq \frac{n-4}{2}$ and there exists at most one
vertex of degree $\frac{n-4}{2}$ in $G'[S]$.

\noindent{\itshape Proof of Claim $2$.} First, we prove that for
each $u_j\in S_1 \ (1\leq j\leq r)$, $d_{G'[S]}(u_j)\geq
\frac{n-4}{2}$. Assume, to the contrary, that there exists one
vertex $u_p\in S_1$ such that $d_{G'[S]}(u_p)\leq \frac{n-6}{2}$. By
the above procedure, there exists a vertex $u_q\in S_2$ such that
when we pick up the edge $e_i=u_pu_q$ from $G_{i-1}[S]$ the degree
of $u_p$ in $G_{i}[S]$ is equal to $\frac{n-6}{2}$. That is
$d_{G_{i}[S]}(u_p)=\frac{n-6}{2}$ and
$d_{G_{i-1}[S]}(u_p)=\frac{n-4}{2}$. From our procedure,
$|E_G[u_q,S_1]|=|E_{G_{i-1}}[u_q,S_1]|$. Without loss of generality,
let $|E_G[u_q,S_1]|=t$ and $u_qu_j\in E(G)$ for $1\leq j\leq t$; see
Figure 1 $(b)$. Thus $u_p\in \{u_1,u_2,\cdots,u_t\}$. Recall that
$|E_G[u_j,S_1]|\geq 2$ for each $u_j\in S_2 \ (r+1\leq j\leq n-1)$.
Since $u_q\in S_2$, we have $t\geq 2$. Clearly, $u_qu_j\notin E(G)$
and hence $u_qu_j\in M$ for $t+1\leq j\leq r$ by our procedure,
namely, $|E_{K_n[M]}[u_q,S_1]|=r-t$. Since
$d_{G_{i-1}[S]}(u_p)=\frac{n-4}{2}$, by our procedure
$d_{G_{i-1}[S]}(u_j)\leq \frac{n-4}{2}$ for each $u_j\in S_1 \
(1\leq j\leq t)$. Assume, to the contrary, that there is a vertex
$u_s \ (1\leq s\leq t)$ such that $d_{G_{i-1}[S]}(u_s)\geq
\frac{n-2}{2}$. Then we should choose the edge $u_qu_s$ instead of
$e_i=u_qu_p$ by our procedure, a contradiction. We conclude that
$d_{G_{i-1}[S]}(u_j)\leq \frac{n-4}{2}$ for each $u_j\in S_1 \
(1\leq i\leq t)$. Clearly, there are at least $n-2-\frac{n-4}{2}$
edges incident to each $u_j  \ (1\leq j\leq t)$ that belong to
$M\cup \{e_1,e_2,\cdots,e_{i-1}\}$. Since $i\leq n-1-r$, we have
$\sum_{j=1}^pd_{K_n[M]}(u_j)\geq (n-2-\frac{n-4}{2})t-(i-1)\geq
\frac{n}{2}t-(n-2-r)$ and hence $|M|\geq
d_{K_n[M]}(v)+\sum_{j=1}^td_{K_n[M]}(u_j)+|E_{K_n[M]}[u_q,S_1]|\geq
(n-1-r)+\frac{n}{2}t-(n-2-r)+(r-t)=r+1+\frac{n-2}{2}t\geq
\frac{n}{2}+1+\frac{2(n-2)}{2}=\frac{3n-2}{2}$, which contradicts to
$|M|=\frac{3n-6}{2}$.

Next, we consider to prove that there exists at most one vertex of
degree $\frac{n-4}{2}$ in $G'[S]$. Assume, to the contrary, that
there exist two vertices of degree $\frac{n-4}{2}$ in $G'[S]$, say
$u_{p'},u_{p}$. By the above procedure, there exists a vertex
$u_{q'}\in S_2$ such that when we pick up the edge
$e_{i'}=u_{p'}u_{q'}$ from $G_{i'-1}[S]$ the degree of $u_p$ in
$G_{i'}[S]$ is equal to $\frac{n-4}{2}$, that is
$d_{G_{i'}[S]}(u_{p'})=\frac{n-4}{2}$. By the same reason, there
exists a vertex $u_q\in S_2$ such that when we pick up the edge
$e_i=u_pu_q$ from $G_{i-1}[S]$ the degree of $u_p$ in $G_{i}[S]$ is
equal to $\frac{n-4}{2}$, that is, $d_{G_{i}[S]}(u_p)=\frac{n-4}{2}$
and $d_{G_{i-1}[S]}(u_p)=\frac{n-2}{2}$. Without loss of generality,
let $i'<i$. From our procedure,
$|E_G[u_q,S_1]|=|E_{G_{i-1}}[u_q,S_1]|$. Without loss of generality,
let $|E_G[u_q,S_1]|=t$ and $u_qu_j\in E(G)$ for $1\leq j\leq t$; see
Figure 1 $(b)$. Thus $u_p\in \{u_1,u_2,\cdots,u_t\}$. Recall that
$|E_G[u_j,S_1]|\geq 2$ for each $u_j\in S_2 \ (r+1\leq j\leq n-1)$.
Since $u_q\in S_2$, we have $t\geq 2$. Then $u_qu_j\notin E(G)$ and
hence $u_qu_j\in M$ for $t+1\leq j\leq r$ by our procedure, namely,
$|E_{K_n[M]}[u_q,S_1]|=r-t$. Since
$d_{G_{i-1}[S]}(u_p)=\frac{n-2}{2}$, by our procedure
$d_{G_{i-1}[S]}(u_j)\leq \frac{n-2}{2}$ for each $u_j\in S_1 \
(1\leq j\leq t)$. Assume, to the contrary, that there is a vertex
$u_s \ (1\leq s\leq t)$ such that $d_{G_{i-1}[S]}(u_s)\geq
\frac{n}{2}$. Then we should choose the edge $u_qu_s$ instead of
$e_i=u_qu_p$ by our procedure, a contradiction. We conclude that
$d_{G_{i-1}[S]}(u_j)\leq \frac{n-2}{2}$ for each $u_j\in S_1 \
(1\leq i\leq t)$. If $u_{p'}\in \{u_1,\cdots,u_t\}$, without loss of
generality, let $u_{p'}=u_1$, then
$d_{K_n[M]}(u_1)+\sum_{j=2}^td_{K_n[M]}(u_j)\geq
(n-2-d_{G_{i-1}[S]}(u_1))+(n-2-\frac{n-2}{2})(t-1)-(i-1)\geq
(n-2-d_{G_{i'}[S]}(u_1))+\frac{n-2}{2}(t-1)-(i-1)\geq
(n-2-\frac{n-4}{2})+\frac{n-2}{2}(t-1)-(n-2-r)=\frac{n-2}{2}t-n+3+r$
since $i\leq n-1-r$. Since $t\geq 2$ and $r\geq \frac{n}{2}$, we
have $|M|\geq
d_{K_n[M]}(v)+d_{K_n[M]}(u_1)+\sum_{j=2}^td_{K_n[M]}(u_j)+|E_{K_n[M]}[u_q,S_1]|\geq
(n-1-r)+(\frac{n-2}{2}t-n+3+r)+(r-t)=\frac{n-4}{2}t+r+2\geq
\frac{2(n-4)}{2}+\frac{n}{2}+2\geq \frac{3n-4}{2}$, which
contradicts to $|M|=\frac{3n-6}{2}$. If $u_{p'}\notin
\{u_1,\cdots,u_t\}$, then $u_{p'}\in \{u_{t+1},\cdots,u_r\}$ and
$d_{K_n[M]}(u_{p'})+\sum_{j=1}^td_{K_n[M]}(u_j)\geq
(n-2-d_{G_{i-1}[S]}(u_{p'}))+(n-2-\frac{n-2}{2})t-(i-1)\geq
(n-2-d_{G_{i'}[S]}(u_{p'}))+\frac{n-2}{2}t-(i-1)\geq
(n-2-\frac{n-4}{2})+\frac{n-2}{2}t-(n-2-r)=\frac{n-2}{2}(t+1)-n+3+r$
since $i\leq n-1-r$. Since $t\geq 2$ and $r\geq \frac{n}{2}$, we
have $|M|\geq
d_{K_n[M]}(v)+d_{K_n[M]}(u_{p'})+\sum_{j=1}^pd_{K_n[M]}(u_j)+(|E_{K_n[M]}[u_q,S_1]|-1)\geq
(n-1-r)+\frac{n-2}{2}(t+1)-n+3+r+(r-1-t)=r+1+\frac{n-4}{2}t+\frac{n-2}{2}\geq
\frac{n}{2}+1+\frac{2(n-4)}{2}+\frac{n-2}{2}=2n-4$, which
contradicts to $|M|=\frac{3n-6}{2}$. The proof of this claim is
complete.

From Claim $2$, $d_{G'[S]}(u_j)\geq \frac{n-4}{2}$ for each $u_j\in
S_1 \ (1\leq i\leq r)$ and and there exists at most one vertex of
degree $\frac{n-4}{2}$ in $G'[S]$. For each $u_j\in S_2 \ (r+1\leq
j\leq n-1)$, $d_{G'[S]}(u_j)=d_{G[S]}(u_j)-1=d_{G}(u_j)-1\geq
\delta(G)-1\geq \frac{n-2}{2}$. So $\delta(G'[S])\geq \frac{n-4}{2}$
and there exists at most one vertex of degree $\frac{n-4}{2}$ in
$G'[S]$. Combining this with
$e(G'[S])=e(G)-(n-1)={{n-2}\choose{2}}-\frac{n-2}{2}$, $G'[S]$
contains $\frac{n-4}{2}$ spanning trees by $(2)$ of Lemma
\ref{lem7}. These trees together with the tree $T$ are
$\frac{n-2}{2}$ trees connecting $S$, namely,
$\overline{\kappa}_{n-1}(G)\geq \frac{n-2}{2}$.
\end{proof}

\begin{proposition}\label{pro3}
For a connected graph $G$ of order $n\ (n\geq 11)$,
$\overline{\kappa}_{n-1}(G)=\lfloor\frac{n-1}{2}\rfloor$ if and only
if $G=K_n\setminus M$ and $M\subseteq E(K_n)$ satisfies one of the
following conditions:

$\bullet$ $1\leq |M|\leq n-2$ for $n$ odd;

$\bullet$ $\frac{n}{2}\leq |M|\leq n$ for $n$ even;

$\bullet$ $n+1\leq |M|\leq \frac{3n-6}{2}$ and $d_G(u_1)\geq
\frac{n-2}{2}$ where $u_1$ is a second minimal degree vertex in $G$
for $n$ even.
\end{proposition}
\begin{proof}
For $n$ odd, if $G$ is a connected graph of order $n$ such that
$\overline{\kappa}_{n-1}(G)=\frac{n-1}{2}$, then we can consider $G$
as the graph obtained from a complete graph $K_{n}$ by deleting some
edges. Set $G=K_n\setminus M$ where $M\subseteq E(K_n)$. From
Proposition \ref{pro1}, $|M|\geq 1$. Combining this with $(3)$ of
Lemma \ref{lem6}, $1\leq |M|\leq n-2$. For $n$ even, if $G$ is a
connected graph of order $n$ such that
$\overline{\kappa}_{n-1}(G)=\frac{n-2}{2}$, then we let
$G=K_n\setminus M$, where $M\subseteq E(K_n)$. From Proposition
\ref{pro1}, $|M|\geq \frac{n}{2}$. Combining this with $(1)$ of
Lemma \ref{lem6}, $\frac{n}{2}\leq |M|\leq \frac{3n-6}{2}$.
Furthermore, for $n+1\leq |M|\leq \frac{3n-6}{2}$ we have
$d_G(u_1)\geq \frac{n-2}{2}$ by $(2)$ of Lemma \ref{lem6}, where
$u_1$ is a second minimal degree vertex. So $\frac{n}{2}\leq |M|\leq
n$, or $n+1\leq |M|\leq \frac{3n-6}{2}$ and $d_G(u_1)\geq
\frac{n-2}{2}$.

Conversely, assume that $G$ is a graph satisfying one of the
conditions of this proposition. Then we will show
$\overline{\kappa}_{n-1}(G)=\lfloor\frac{n-1}{2}\rfloor$. For $n$
odd, $G=K_n\setminus M$ and $M\subseteq E(K_n)$ such that $1\leq
|M|\leq n-2$. In fact, we only need to show that
$\overline{\kappa}_{n-1}(G)\geq \lfloor\frac{n-1}{2}\rfloor$ for
$|M|=n-2$. It follows by Lemma \ref{lem8}. Combining with
Proposition \ref{pro1}, $\overline{\kappa}_{n-1}(G)=
\lfloor\frac{n-1}{2}\rfloor$. For $n$ even, $G=K_n\setminus M$ and
$M\subseteq E(K_n)$ such that $\frac{n}{2}\leq |M|\leq n$, or
$n+1\leq |M|\leq \frac{3n-6}{2}$ and $d_G(u_1)\geq \frac{n-2}{2}$
where $u_1$ is a second minimal degree vertex. Actually, for
$\frac{n}{2}\leq |M|\leq n$, we claim that $d_G(u_1)\geq
\frac{n-2}{2}$, where $u_1$ is a second minimal degree vertex.
Otherwise, let $d_G(u_1)\leq \frac{n-4}{2}$. Let $v$ be the vertex
such that $d_G(v)=\delta(G)$. From the definition of the second
minimal degree vertex, $d_G(v)\leq d_G(u_1)\leq \frac{n-4}{2}$ and
hence $d_{K_n[M]}(v)\geq d_{K_n[M]}(u_1)\geq
n-1-\frac{n-4}{2}=\frac{n+2}{2}$. Therefore, $|M|\geq
d_{K_n[M]}(v)+d_{K_n[M]}(u_1)\geq n+2$, a contradiction. So we only
need to show that $\overline{\kappa}_{n-1}(G)\geq
\lfloor\frac{n-1}{2}\rfloor$ for $|M|=\frac{3n-6}{2}$ and
$d_G(u_1)\geq \frac{n-2}{2}$ where $u_1$ is a second minimal degree
vertex. It follows by Lemma \ref{lem9}. From this together with
Proposition \ref{pro1}, $\overline{\kappa}_{n-1}(G)=
\lfloor\frac{n-1}{2}\rfloor$.
\end{proof}

\begin{proposition}\label{pro4}
For a connected graph $G$ of order $n \ (n\geq 11)$,
$\overline{\lambda}_{n-1}(G)=\lfloor\frac{n-1}{2}\rfloor$ if and
only if $G=K_n\setminus M$ and $M\subseteq E(K_n)$ satisfies one of
the following conditions.

$\bullet$ $1\leq |M|\leq n-2$ for $n$ odd;

$\bullet$ $\frac{n}{2}\leq |M|\leq n$ for $n$ even;

$\bullet$ $n+1\leq |M|\leq \frac{3n-6}{2}$ and $d_G(u_1)\geq
\frac{n-2}{2}$ where $u_1$ is a second minimal degree vertex in $G$
for $n$ even.
\end{proposition}
\begin{proof}
Assume that $G$ is a connected graph satisfying the conditions of
Proposition \ref{pro4}. From Observation \ref{obs1} and Proposition
\ref{pro3}, it follows that $\overline{\lambda}_{n-1}(G)\geq
\overline{\kappa}_{n-1}(G)=\lfloor\frac{n-1}{2}\rfloor$. Combining
this with Proposition \ref{pro2},
$\overline{\lambda}_{n-1}(G)=\lfloor\frac{n-1}{2}\rfloor$.
Conversely, if
$\overline{\lambda}_{n-1}(G)=\lfloor\frac{n-1}{2}\rfloor$, then from
Lemma \ref{lem6} we have $G=K_n\setminus M$ for $n$ odd,  where $M$
is an edge set such that $1\leq |M|\leq n-2$; $G=K_n\setminus M$ for
$n$ even, where $M$ is an edge set such that $\frac{n}{2}\leq
|M|\leq n$, or $n+1\leq |M|\leq \frac{3n-6}{2}$ and $d_G(u_1)\geq
\frac{n-2}{2}$.
\end{proof}

\subsection{The subcase $1\leq \ell\leq
\lfloor\frac{n-5}{2}\rfloor$}

Now we consider the case $1\leq \ell\leq
\lfloor\frac{n-5}{2}\rfloor$.

\begin{lemma}\label{lem10}
Let $H$ is a connected graph of order $n-1\ (n\geq 12)$. If $e(H)=
{{n-2}\choose{2}}+2\ell-(n-1) \ (1\leq \ell\leq
\lfloor\frac{n-5}{2}\rfloor)$ and $\delta(H)\geq \ell$ and any two
vertices of degree $\ell$ are nonadjacent, then $H$ contains $\ell$
edge-disjoint spanning trees.
\end{lemma}

\begin{proof}
Let $\mathscr{P}=\bigcup_{i=1}^pV_i$ be a partition of $V(G)$ with
$|V_i|=n_i \ (1\leq i\leq p)$, and $\mathcal {E}_p$ be the set of
edges between distinct blocks of $\mathscr{P}$ in $G$. It suffices
to show $|\mathcal {E}_p|\geq \ell (|\mathscr{P}|-1)$ so that we can
use Theorem \ref{th1}.

The case $p=1$ is trivial, thus we assume $p\geq 2$. For $p=2$, we
have $\mathscr{P}=V_1\cup V_2$. Set $|V_1|=n_1$. Then
$|V_2|=n-1-n_1$. If $n_1=1$ or $n_1=n-2$, then $|\mathcal
{E}_2|=|E_G[V_1,V_2]|\geq \ell$ since $\delta(H)\geq \ell$. If
$n_1=2$ or $n_1=n-3$, then $|\mathcal {E}_2|=|E_G[V_1,V_2]|\geq
\ell$ since $\delta(H)\geq \ell$ and any two vertices of degree
$\ell$ are nonadjacent. Suppose $3\leq n_1\leq n-4$. Then $|\mathcal
{E}_2|=|E_G[V_1,V_2]|\geq
{{n-2}\choose{2}}+2\ell-(n-1)-{{n_1}\choose{2}}-{{n-1-n_1}\choose{2}}
=-n_1^2+(n-1)n_1+2\ell-(2n-3)$. Since $3\leq n_1 \leq n-4$, one can
see that $|\mathcal {E}_2|$ attains its minimum value when $n_1=3$
or $n_1=n-4$. Thus $|\mathcal {E}_2|\geq n-9+2\ell\geq \ell$. So the
conclusion holds for $p=2$ by Theorem \ref{th1}.

Consider the case $p=3$. We will show $|\mathcal {E}_3|\geq 2\ell$.
Let $\mathscr{P}=V_1\cup V_2\cup V_3$ and $|V_i|=n_i \ (i=1,2,3)$
where $n_1+n_2+n_3=n-1$. If there are two of $n_1,n_2,n_3$ that
equals 1, say $n_1=n_2=1$, then $|\mathcal {E}_3|\geq 2\ell$ since
$\delta(H)\geq \ell$ and any two vertices of degree $\ell$ are
nonadjacent. If there is at most one of $n_1,n_2,n_3$ that equals
$1$, then we need to prove that $|\mathcal {E}_3|\geq
{{n-2}\choose{2}}+2\ell-(n-1)-\sum_{i=1}^3{{n_i}\choose{2}}\geq
2\ell$. Since $f(n_1,n_2,n_3)=\sum_{i=1}^3{{n_i}\choose{2}}$ attains
its maximum value when $n_1=1$, $n_2=2$ and $n_3=n-4$, we need the
inequality ${{n-2}\choose{2}}+2\ell-(n-1)-{{n-4}\choose{2}}-1\geq
2\ell$. Since $n\geq 12$, the inequality holds. So the conclusion
holds for $p=3$ by Theorem \ref{th1}. For $p=n-1$, we will show
$|\mathcal {E}_{n-1}|\geq \ell(n-2)$ so that we can use Theorem
\ref{th1}. That is ${{n-2}\choose{2}}+2\ell-(n-1)\geq \ell(n-2)$.
Thus we need the inequality $(n-2-2\ell)(n-4)-n\geq 0$. Since
$\ell\leq \lfloor\frac{n-5}{2}\rfloor$, the inequality holds. For
$p=n-2$, we need to prove $|\mathcal {E}_{n-2}|\geq \ell(n-3)$.
Clearly, $|\mathcal {E}_{n-2}|\geq
{{n-2}\choose{2}}+2\ell-(n-1)-1\geq \ell(n-3)$. Thus we need the
inequality $(n-2-2\ell)(n-5)-4\geq 0$. Since $\ell\leq
\lfloor\frac{n-5}{2}\rfloor$, this inequality holds.

Let us consider the remaining case $p$ with $4\leq p\leq n-4$.
Clearly, we need to prove that $|\mathcal {E}_p|\geq
{{n-2}\choose{2}}+2\ell-(n-1)-\sum_{i=1}^p{{n_i}\choose{2}}\geq
\ell(p-1)$, that is, $\frac{(n-2)(n-3)}{2}+2\ell-(n-1)-\ell p+\ell
\geq \sum_{i=1}^p{{n_i}\choose{2}}$. Since
$f(n_1,n_2,\cdots,n_p)=\sum_{i=1}^p{{n_i}\choose{2}}$ achieves its
maximum value when $n_1=n_2=\cdots=n_{p-1}=1$ and $n_p=n-p$, we need
the inequality $\frac{(n-2)(n-3)}{2}+3\ell-(n-1)-\ell p\geq
\frac{(n-p)(n-p-1)}{2}$. It is equivalent to
$(2n-2\ell-p-4)(p-3)\geq 4$. One can see that the inequality holds
since $\ell\leq \frac{n-5}{2}$ and $4\leq p\leq n-4$. From Theorem
\ref{th1}, we know that there exist $\ell$ edge-disjoint spanning
trees.
\end{proof}

\begin{lemma}\label{lem11}
Let $G$ be a connected graph of order $n \ (n\geq 12)$. If $e(G)\geq
{{n-2}\choose{2}}+2\ell  \ (1\leq \ell\leq
\lfloor\frac{n-5}{2}\rfloor)$, $\delta(G)\geq \ell+1$ and any two
vertices of degree $\ell+1$ are nonadjacent, then
$\overline{\kappa}_{n-1}(G)\geq \ell+1$.
\end{lemma}
\begin{proof}
The following claim can be easily proved.

\noindent {\bf Claim 3}. $\Delta(G)\geq n-4$.

\noindent{\itshape Proof of Claim $3$.} Assume, to the contrary,
that $\Delta(G)\leq n-5$. Then $(n-2)(n-3)+4\ell=2e(G)\leq
n\Delta(G)\leq n(n-5)$, which implies that $4\ell+6\leq 0$, a
contradiction.

From Claim $3$, $n-4\leq \Delta(G)\leq n-1$. Our basic idea is to
find out a Steiner tree $T$ connecting $S=V(G)\setminus v$, where
$v\in V(G)$ such that $d_G(v)=\Delta(G)$. Let $G_1=G\setminus E(T)$.
Then we prove that $G_1[S]$ satisfies the conditions of Lemma
\ref{lem10} so that $G_1[S]$ contains $\ell$ spanning trees. These
trees together with the tree $T$ are $\ell+1$ internally disjoint
trees connecting $S$, which implies that
$\overline{\kappa}_{n-1}(G)\geq \ell+1$, as desired. We distinguish
the following four cases to show this lemma.

If $\Delta(G)=n-1$, then there exists a vertex $v\in V(G)$ such that
$d_G(v)=n-1$. Let $S=V(G)\setminus v=\{u_1,u_2,\cdots,u_{n-1}\}$.
Then $T=u_1v\cup u_2v\cup \cdots\cup u_{n-1}v$ is a tree connecting
$S$. Set $G_1=G\setminus E(T)$. Since $\delta(G)\geq \ell+1$ and any
two vertices of degree $\ell+1$ are nonadjacent, it follows that
$\delta(G_1[S])\geq \ell$ and any two vertices of degree $\ell$ are
nonadjacent. From Lemma \ref{lem10}, $G_1[S]$ contains $\ell$
spanning trees, as desired.

Consider the case $\Delta(G)=n-4$. We claim that $\delta(G)\geq
\ell+4$. Otherwise, let $\delta(G)\leq \ell+3$. Then there exists a
vertex $u$ such that $d_G(u)\leq \ell+3$. Then
$2[{{n-2}\choose{2}}+2\ell]=2e(G)=\sum_{u\in V(G)}d(u)\leq
d_G(u)+(n-1)\Delta(G)\leq (\ell+3)+(n-1)(n-4)$, which results in
$\ell\leq \frac{1}{3}$, a contradiction. Since $\Delta(G)=n-4$,
there exists a vertex $v\in V(G)$ such that $d_G(v)=n-4$. Let
$S=V(G)\setminus v=\{u_1,\cdots,u_{n-1}\}$ such that $vu_{n-1},
vu_{n-2}, vu_{n-3}\notin E(G)$. Pick up $u_i\in N_G(u_{n-1}), u_j\in
N_G(u_{n-2}), u_k\in N_G(u_{n-3})$ (note that $u_i,u_j,u_k$ are not
necessarily different). Then the tree $T=vu_1\cup vu_2\cup
\cdots\cup vu_{n-4}\cup u_iu_{n-1}\cup u_ju_{n-2}\cup u_ku_{n-1}$ is
our desired. Set $G_1=G\setminus E(T)$. Since $\delta(G)\geq
\ell+4$, $G_1[S]$ contains at most one vertex of degree $\ell$, as
desired.

If $\Delta(G)=n-2$, then there exists a vertex of degree $n-2$ in
$G$, say $v$. Let $S=G\setminus v=\{u_1,u_2,\cdots,u_{n-1}\}$ such
that $u_{n-1}$ is the unique vertex with $u_{n-1}v\notin E(G)$. Let
$d_G(u_{n-1})=x$. Without loss of generality, let
$N_G(u_{n-1})=\{u_1,\cdots,u_{x}\}$. Since $\delta(G)\geq \ell+1$,
$x\geq \ell+1\geq 2$. First, we consider the case $x\geq 3$. We
claim that there exists a vertex, say $u_i \ (1\leq i\leq x)$, such
that $d_G(u_i)\geq \ell+3$. Otherwise, let $d_G(u_{j})\leq \ell+2$
for each $u_j \ (1\leq j\leq x)$. Then $(n-2)(n-3)+4\ell=2e(G)\leq
d_G(u_{n-1})+d_G(v)+\sum_{j=1}^xd_G(u_{j})+\sum_{j=x+1}^{n-2}d_G(u_{j})\leq
x+(n-2)+(\ell+2)x+(n-2-x)(n-2)$ and hence $x\leq
\frac{2n-4\ell-4}{n-\ell-5}$. Since $x\geq 3$, $n+\ell-11\leq 0$,
which contradicts to $n\geq 12$. So there exists a vertex, say $u_i
\ (1\leq i\leq x)$, such that $d_G(u_i)\geq \ell+3$. Then the tree
$T=vu_1\cup vu_2\cup \cdots\cup vu_{n-2}\cup u_{n-1}u_i$ is our
desired. Set $G_1=G\setminus E(T)$. It is clear that
$\delta(G_1[S])\geq \ell$ and any two vertices of degree $\ell$ are
nonadjacent, as desired. Next, we consider the case $x=2$. Then
$\ell=1$, $d_G(u_{n-1})=2$ and $N_G(u_{n-1})=\{u_1,u_2\}$. Let $p$
be the number of vertices of degree $2$ in $G$. We claim $0\leq
p\leq 3$. Otherwise, let $p\geq 4$. Then
$2{{n-2}\choose{2}}+4=2e(G)=\sum_{v\in V(G)}d(v)\leq 2p+(n-p)(n-2)$
and hence $p\leq \frac{3n-10}{n-4}$. Since $p\geq 4$, it follows
that $n\leq 6$, a contradiction. So $0\leq p\leq 3$. If $p=3$, then
there are three vertices of degree $2$, say $v_1,v_2,v_3$. Let
$G_1=G\setminus \{v_1,v_2,v_3\}$. Since the three vertices are
pairwise nonadjacent, $|V(G_1)|=n-3$ and
$e(G_1)={{n-2}\choose{2}}+2-6
={{n-2}\choose{2}}-4>{{n-3}\choose{2}}$, a contradiction. So we can
assume $0\leq p\leq 2$. If $p=2$, then there are two vertices of
degree $2$, say $v_1,v_2$. Let $G_1=G\setminus \{v_1,v_2\}$. Then
$G_1$ is a graph obtained from a clique of order $n-2$ by deleting
$2$ edges and hence $\overline{\kappa}_{n-2}(G_1)\geq
\lfloor\frac{n-2}{2}\rfloor-2\geq 2$, that is, $G_1$ contains two
spanning trees, say $T_1',T_2'$. Let $N_G(v_1)=\{u_1,u_2\}$, the
trees $T_i=T_i'\cup v_1u_i \ (i=1,2)$ are two Steiner trees
connecting $S=V(G)\setminus v_2$, which implies that
$\overline{\kappa}_{n-1}(G)\geq 2$. So we now assume $0\leq p\leq
1$. Consider the case $p=1$. If $d_G(u_{n-1})=2$, then $d_G(u_j)\geq
3$ for each $u_j \ (1\leq j\leq n-2)$. Recall that
$N_G(u_{n-1})=\{u_1,u_{2}\}$, certainly we have $d_G(u_j)\geq 3 \
(j=1,2)$. Then the tree $T=vu_1\cup vu_2\cup \cdots\cup vu_{n-2}\cup
u_1u_{n-1}$ is a Steiner tree connecting $S=V(G)\setminus v$. Set
$G_1=G\setminus E(T)$. Clearly, $d_{G_1[S]}(u_1)\geq 1$,
$d_{G_1[S]}(u_{n-1})=1$ and $u_1u_{n-1}\notin E(G_1[S])$. In
addition, the degree of the other vertices in $G_1[S]$ is at least
$2$, as desired. Assume $d_G(u_{n-1})\geq 3$. Let $u_i$ be the
vertex of degree $2$ in $V(G)\setminus \{v,u_{n-1}\}$. If $u_i\in
N_G(u_{n-1})$, then there is another vertex $u_j\in N_G(u_{n-1})$
such that $d_G(u_j)\geq 3$ since $p=1$. Then the tree $T=vu_1\cup
vu_2\cup \cdots\cup vu_{n-2}\cup u_ju_{n-1}$ is our desired. Set
$G_1=G\setminus E(T)$. Obviously, $d_{G_1[S]}(u_i)=1$,
$d_{G_1[S]}(u_j)\geq 1$, $d_{G_1[S]}(u_{n-1})\geq 2$,
$u_iu_{j}\notin E(G_1[S])$ and the degree of the other vertices in
$G_1[S]$ is at least $2$, as desired. If $u_i\notin N_G(u_{n-1})$,
then there exists a vertex $u_j\in N_G(u_{n-1})$ such that
$d_G(u_j)\geq 3$ and $u_iu_{j}\notin E(G)$. Thus the tree
$T=vu_1\cup vu_2\cup \cdots\cup vu_{n-2}\cup u_ju_{n-1}$ is our
desired. Set $G_1=G\setminus E(T)$. Clearly, $d_{G_1[S]}(u_i)=1$,
$d_{G_1[S]}(u_t)\geq 1$, $d_{G_1[S]}(u_{n-1})\geq 2$,
$u_iu_{j}\notin E(G_1[S])$ and the degree of the other vertices in
$G_1[S]$ is at least $2$, as desired. For the remaining case $p=0$,
we choose a vertex $u_j\in N_G(u_{n-1})$ and the tree $T=vu_1\cup
vu_2\cup \cdots\cup vu_{n-2}\cup u_ju_{n-1}$ is our desired. Set
$G_1=G\setminus E(T)$. Clearly, $\delta(G_1[S])\geq 1$ and there is
at most one vertex of degree $1$, as desired.

Let us consider the remaining case $\Delta(G)=n-3$. Then there
exists a vertex of degree $n-3$, say $v$. Let $p$ be the number of
vertices of degree $\ell+1$. Since $(n-2)(n-3)+4\ell=2e(G)\leq
p(\ell+1)+(n-p)(n-3)$, it follows that $p\leq
\frac{2n-4\ell-6}{n-\ell-4}$. Consider the case $\ell\geq 2$. Since
$p\leq \frac{2n-4\ell-6}{n-\ell-4}$, if $p\geq 2$ then $\ell\leq 1$,
a contradiction. So $0\leq p\leq 1$ for $2\leq \ell\leq
\lfloor\frac{n-5}{2}\rfloor$. Let $V(G)\setminus
v=\{u_1,\cdots,u_{n-1}\}$ such that $vu_{n-1},vu_{n-2}\notin E(G)$.
Without loss of generality, let $d_G(u_{n-1})\geq d_G(u_{n-2})$. For
vertex $u\in V(G)$, we choose $\ell+1$ vertices in $N_G(u)$, say
$u_1,u_2,\cdots,u_{\ell+1}$ and the following claim can be easily
proved.

\noindent {\bf Claim 4}. For $\ell\geq 2$, there exists a vertex
$u_i\in N_G(u)$ such that $d_G(u_i)\geq \ell+4 \ (1\leq i\leq
\ell+1)$.

\noindent{\itshape Proof of Claim $4$.} Assume, to the contrary,
that $d_G(u_j)\leq \ell+3$ for each $u_j \ (1\leq j\leq \ell+1)$.
Then $(n-2)(n-3)+4\ell=2e(G)\leq (\ell+1)(\ell+3)+(n-\ell-1)(n-3)$
and hence $(\ell-1)(n-3)\leq \ell^2+3$. So $n-3\leq
\frac{\ell^2+3}{\ell-1}=\ell+1+\frac{4}{\ell-1}\leq \ell+5\leq
\frac{n+5}{2}$, which contradicts to $n\geq 12$.

First, we consider the case $u_{n-1}u_{n-2}\in E(G)$. From the
above, $0\leq p\leq 1$ for $2\leq \ell\leq
\lfloor\frac{n-5}{2}\rfloor$, that is, there is at most one vertex
of degree $\ell+1$ in $G$. If $d_G(u_{n-2})=\ell+1$, then
$d_G(u_{n-1})\geq \ell+2$ and hence there exists a vertex $u_i\in
N_{G}(u_{n-1})\setminus u_{n-2}$ such that $d_G(u_i)\geq \ell+4$ by
Claim $4$. Then the tree $T=vu_1\cup vu_2\cup \cdots\cup
vu_{n-3}\cup u_iu_{n-1}\cup u_{n-1}u_{n-2}$ is a Steiner tree
connecting $S=V(G)\setminus v$. Clearly, $d_{G_1[S]}(u_{n-1})\geq
d_G(u_{n-1})-2\geq \ell$, $d_{G_1[S]}(u_{n-2})=d_G(u_{n-2})-1=\ell$
and $u_{n-2}u_{n-1}\notin E(G_1)$. In addition,
$d_{G_1[S]}(u_{i})\geq d_G(u_{i})-2\geq \ell+2$ and
$d_{G_1[S]}(u_{j})\geq d_G(u_{j})-1\geq \ell+1$ for each $u_j\in
V(G)\setminus \{u_{n-1},u_{n-2},u_{i},v\}$. Thus $\delta(G_1[S])\geq
\ell$ and any two vertices of degree $\ell$ are nonadjacent, as
desired. If $d_G(u_{n-2})\geq \ell+2$, then $d_G(u_{n-1})\geq
d_G(u_{n-2})\geq \ell+2$. From Claim $4$, there exist two vertices,
say $u_i,u_j$, such that $u_i\in N_G(u_{n-1})\setminus u_{n-2}$,
$u_j\in N_G(u_{n-2})\setminus u_{n-1}$, $d_G(u_{i})\geq \ell+4$ and
$d_G(u_{j})\geq \ell+4$ (note that $u_i,u_j$ are not necessarily
different). Then the tree $T=vu_1\cup vu_2\cup \cdots\cup
vu_{n-3}\cup u_iu_{n-1}\cup u_{j}u_{n-2}$ is our desired. Set
$G_1=G\setminus E(T)$. One can see that $G_1[S]$ satisfies the
conditions of Lemma \ref{lem10}. So we can get
$\overline{\kappa}_{n-1}(G)\geq \ell+1$. Next, we consider the case
$u_{n-1}u_{n-2}\notin E(G)$. Then $d_G(u_{n-1})\geq d_G(u_{n-2})\geq
\ell+1$. From Claim $4$, there exist two vertices, say $u_i,u_j$,
such that $u_i\in N_G(u_{n-1})\setminus u_{n-2}$, $u_j\in
N_G(u_{n-2})\setminus u_{n-1}$, $d_G(u_{i})\geq \ell+4$ and
$d_G(u_{j})\geq \ell+4$ (note that $u_i,u_j$ are not necessarily
different). Thus the tree $T=vu_1\cup vu_2\cup \cdots\cup
vu_{n-3}\cup u_iu_{n-1}\cup u_{j}u_{n-2}$ is our desired. Set
$G_1=G\setminus E(T)$ and $S=V(G)\setminus v$. One can check that
$\delta(G_1[S])\geq \ell$ and there is at most one vertex of degree
$\ell$, as desired. Similar to the proof of the case
$\Delta(G)=n-2$, we can prove that the conclusion holds for
$\ell=1$. The proof is now complete.
\end{proof}

\subsection{Results for the maximum generalized local
(edge-)connectivity}

Let $\mathcal {H}_n$ be a graph class obtained from the complete
graph of order $n-2$ by adding two nonadjacent vertices and joining
each of them to any $\ell$ vertices of $K_{n-2}$. The following
theorem summarizes the results for a general $\ell$.

\begin{theorem}\label{th4}
Let $G$ be a connected graph of order $n \ (n\geq 12)$. If
$\overline{\kappa}_{n-1}(G)\leq \ell \ (1\leq \ell\leq
\lfloor\frac{n+1}{2}\rfloor)$, then
$$
e(G)\leq \left\{
\begin{array}{ll}
{{n-2}\choose{2}}+2\ell, &if~1\leq \ell\leq
\lfloor\frac{n-5}{2}\rfloor;\\
{{n-2}\choose{2}}+n-2,&if~\ell=
\lfloor\frac{n-3}{2}\rfloor~and~$n$~is~odd;\\
{{n-2}\choose{2}}+n-4,&if~\ell=
\lfloor\frac{n-3}{2}\rfloor~and~$n$~is~even;\\
{{n-1}\choose{2}}+n-2,&if~\ell=
\lfloor\frac{n-1}{2}\rfloor~and~$n$~is~odd;\\
{{n-1}\choose{2}}+\frac{n-2}{2},&if~\ell=
\lfloor\frac{n-1}{2}\rfloor~and~$n$~is~even;\\
{{n}\choose{2}},&if~\ell=\lfloor\frac{n+1}{2}\rfloor.
\end{array}
\right.
$$
with equality if and only if $G\in \mathcal{H}_n$ for $1\leq
\ell\leq \lfloor\frac{n-5}{2}\rfloor$; $G=K_n\setminus M$ where
$|M|=n-1$ for $\ell=\lfloor\frac{n-3}{2}\rfloor$ and $n$ odd; $G\in
\mathcal{H}_n$ for $\ell=\lfloor\frac{n-3}{2}\rfloor$ and $n$ even;
$G=K_n\setminus e$ where $e\in E(K_n)$ for
$\ell=\lfloor\frac{n-1}{2}\rfloor$ and $n$ odd; $G=K_n\setminus M$
where $|M|=\frac{n}{2}$ for $\ell=\lfloor\frac{n-1}{2}\rfloor$ and
$n$ even; $G=K_n$ for $\ell=\lfloor\frac{n+1}{2}\rfloor$.
\end{theorem}
\begin{proof}
For $1\leq \ell\leq \lfloor\frac{n-5}{2}\rfloor$, we assume that
$e(G)\geq {{n-2}\choose{2}}+2\ell+1$. Then the following claim is
immediate.

\noindent {\bf Claim 5}. $\delta(G)\geq \ell+1$.

\noindent{\itshape Proof of Claim $5$.} Assume, to the contrary,
that $\delta(G)\leq \ell$. Then there exists a vertex $v\in V(G)$
such that $d_G(v)=\delta(G)\leq \ell$, which results in $e(G-v)\geq
e(G)-\ell\geq {{n-2}\choose{2}}+\ell+1$. Since $1\leq \ell\leq
\lfloor\frac{n-5}{2}\rfloor$, it follows that
$\overline{\kappa}_{n-1}(G-v)\geq \ell+1$ by Theorem \ref{th2},
which results in $\overline{\kappa}_{n-1}(G)\geq \ell+1$, a
contradiction.

From Claim $5$, $\delta(G)\geq \ell+1$. If any two vertices of
degree $\ell+1$ are nonadjacent, then
$\overline{\kappa}_{n-1}(G)\geq \ell+1$ by Lemma \ref{lem11}, a
contradiction. Assume that $v_1$ and $v_2$ are two vertices of
degree $\ell+1$ such that $v_1v_2\in E(G)$. Let $G_1=G\setminus
\{v_1,v_2\}$ and $V(G_1)=\{u_1,\cdots,u_{n-1}\}$. Then $e(G_1)\geq
e(G)-(2\ell+1)={{n-2}\choose{2}}$ and hence $G_1$ is a clique of
order $n-2$. Then $G_1$ contains $\lfloor\frac{n-2}{2}\rfloor\geq
\ell+1$ edge-disjoint spanning trees, say
$T_1',T_2',\cdots,T_{\ell+1}'$. Without loss of generality, let
$N_G(v_1)=\{u_1,u_2,\cdots,u_{\ell},v_2\}$. Choose
$S=\{u_1,u_2,\cdots,u_{n-1},v_1\}$. Then $T_i=T_i'\cup v_1u_i \
(1\leq i\leq \ell)$ together with $T_{\ell+1}=T_{\ell+1}'\cup
v_1v_2\cup v_2u_t$ are $\ell+1$ internally disjoint trees connecting
$S$ where $u_t\in N_G(v_2)\setminus v_1$, which implies that
$\overline{\kappa}_{n-1}(G)\geq \ell+1$, a contradiction. So
$e(G)\leq {{n-2}\choose{2}}+2\ell$ for $1\leq \ell\leq
\lfloor\frac{n-5}{2}\rfloor$. From Proposition \ref{pro3}, $e(G)\leq
{{n-2}\choose{2}}+n-2$ for $\ell=\lfloor\frac{n-3}{2}\rfloor$ and
$n$ odd, and $e(G)\leq {{n-2}\choose{2}}+n-4$ for
$\ell=\lfloor\frac{n-3}{2}\rfloor$ and $n$ even. From Proposition
\ref{pro1}, $e(G)\leq {{n-1}\choose{2}}+n-2$ for
$\ell=\lfloor\frac{n-1}{2}\rfloor$ and $n$ odd, and $e(G)\leq
{{n-1}\choose{2}}+\frac{n-2}{2}$ for
$\ell=\lfloor\frac{n-1}{2}\rfloor$ and $n$ even. If
$\ell=\lfloor\frac{n+1}{2}\rfloor$, then for any connected graph $G$
it follows that $\overline{\kappa}_{n-1}(G)\leq \ell$ by Observation
\ref{obs4} and hence $e(G)\leq {{n}\choose{2}}$.

Now we characterize the graphs attaining these upper bounds. For
$\ell=\lfloor\frac{n+1}{2}\rfloor$, if $e(G)={{n}\choose{2}}$, then
$G=K_n$. For $\ell=\lfloor\frac{n-1}{2}\rfloor$ and $n$ odd, if
$e(G)={{n-1}\choose{2}}+n-2$, then $G=K_n\setminus e$ where $e\in
E(K_n)$. For $\ell=\lfloor\frac{n-1}{2}\rfloor$ and $n$ even, if
$e(G)={{n-1}\choose{2}}+\frac{n-2}{2}$, then $G=K_n\setminus M$
where $|M|=\frac{n}{2}$. For $\ell=\lfloor\frac{n-3}{2}\rfloor$ and
$n$ odd, if $e(G)={{n-2}\choose{2}}+n-2$, then $G=K_n\setminus M$
where $|M|=n-1$. Assume that $e(G)={{n-2}\choose{2}}+n-4$ for
$\ell=\lfloor\frac{n-3}{2}\rfloor$ and $n$ even. From Proposition
\ref{pro3}, $G$ is a graph obtained from the complete graph
$K_{n-2}$ by adding two nonadjacent vertices and adding
$\frac{n-4}{2}$ edges between each of them and the complete graph
$K_{n-2}$, that is, $G\in \mathcal {H}_n$.

Let us now focus on the case $1\leq \ell\leq
\lfloor\frac{n-5}{2}\rfloor$. Suppose
$e(G)={{n-2}\choose{2}}+2\ell$. Similar to Claim $5$, $\delta(G)\geq
\ell$. If $\delta(G)=\ell+1$ and any two vertices of degree $\ell+1$
are nonadjacent, then $\overline{\kappa}_{n-1}(G)\geq \ell+1$ by
Lemma \ref{lem11}, a contradiction. Let $v_1$ and $v_2$ be two
vertices of degree $\ell+1$ such that $v_1v_2\in E(G)$. It is clear
that $G_1=G\setminus \{v_1,v_2\}$ is a graph obtained from the
complete graph of order $n-2$ by deleting an edge. For $n$ odd, from
Corollary \ref{cor1} we have
$\overline{\kappa}_{n-2}(G_1)=\lfloor\frac{n-2}{2}\rfloor=\frac{n-3}{2}\geq
\ell+1$ since $\ell\leq \lfloor\frac{n-5}{2}\rfloor=\frac{n-5}{2}$.
For $n$ even, from Corollary \ref{cor1}, it follows that
$\overline{\kappa}_{n-2}(G_1)\geq \lfloor\frac{n-2}{2}\rfloor-1=
\frac{n-4}{2}\geq \ell+1$ since $\ell\leq
\lfloor\frac{n-5}{2}\rfloor=\frac{n-6}{2}$. Clearly, $G_1$ contains
$\ell+1$ edge-disjoint spanning trees, say
$T_1',T_2',\cdots,T_{\ell+1}'$. Set
$N_G(v_1)=\{u_1,u_2,\cdots,u_{\ell},v_2\}$. Then $T_i=T_i'\cup
v_1u_i \ (1\leq i\leq \ell)$ and $T_{\ell+1}=T_{\ell+1}'\cup
v_1v_2\cup v_1u_t$ are $\ell+1$ internally disjoint trees connecting
$S=V(G)\setminus v_2$ where $u_t\in N_G(v_2)\setminus v_1$, which
implies that $\overline{\kappa}_{n-1}(G)\geq \ell+1$, a
contradiction. Suppose $\delta(G)=\ell$. If there exist two vertices
of degree $\ell$, say $v_1,v_2$, such that $v_1v_2\in E(G)$. Set
$G_1=G\setminus \{v_1,v_2\}$. Then $|V(G_1)|=n-2$ and
$e(G_1)={{n-2}\choose{2}}+1$, a contradiction.

So we assume that any two vertices of degree $\ell$ are nonadjacent
in $G$. Let $p$ be the number of vertices of degree $\ell$. The
following claim can be easily proved.

\noindent {\bf Claim 6}. $2\leq p\leq 3$.

\noindent{\itshape Proof of Claim $6$.} Assume $p\geq 4$. Then
$2{{n-2}\choose{2}}+4\ell=2e(G)=\sum_{v\in V(G)}d(v)\leq
p\ell+(n-p)(n-1)$ and hence $p\leq \frac{4n-4\ell-6}{n-\ell-1}$.
Since $p\geq 4$, it follows that $4n-4\ell-4\leq 4n-4\ell-6$, a
contradiction. Assume $p=1$, that is, $G$ contains exact one vertex
of degree $\ell$, say $v_1$. Set $G_1=G\setminus v_1$. Clearly,
$e(G_1)=e(G)-\ell={{n-2}\choose{2}}+\ell$. Since
$\overline{\kappa}_{n-1}(G)\leq \ell$, it follows that
$\overline{\kappa}_{n-1}(G_1)\leq \overline{\kappa}_{n-1}(G)\leq
\ell$. From Theorem \ref{th2}, $G_1$ is a graph obtained from a
clique of order $n-2$ by adding a vertex of degree $\ell$, say
$v_2$. Since $p=1$ and $v_1v_2\notin E(G)$, we have
$d_G(v_1)=\ell+1$ and $d_G(v_2)=\ell$. Clearly, $G_1=G\setminus
\{v_1,v_2\}$ is a clique of order $n-2$. Thus $G_1$ contains
$\lfloor\frac{n-2}{2}\rfloor\geq \ell+1$ edge-disjoint spanning
trees, say $T_1',T_2',\cdots,T_{\ell+1}'$. Without loss of
generality, let $N_G(v_1)=\{v_2,u_1,u_2,\cdots,u_{\ell}\}$. Then the
trees $T_i=v_1u_i\cup T_i' \ (1\leq i\leq \ell)$ together with
$T_{\ell+1}=T_{\ell+1}'\cup v_1v_2\cup v_2u_t$ form $\ell+1$
edge-disjoint trees connecting $S=V(G)\setminus v_2$, where $u_t\in
N_G(v_2)\setminus v_1$. This implies that
$\overline{\kappa}_{n-1}(G)\geq \ell+1$, a contradiction.

From Claim $6$, we know that $p=2,3$. If $p=3$, then $G$ contains
three vertices of degree $\ell$, say $v_1,v_2,v_3$. Set
$G_1=G\setminus \{v_1,v_2,v_3\}$. Then $|V(G_1)|=n-3$ and
$e(G_1)={{n-2}\choose{2}}+2\ell-3\ell={{n-2}\choose{2}}-\ell>{{n-3}\choose{2}}$
since $1\leq \ell\leq \lfloor\frac{n-5}{2}\rfloor$, a contradiction.
If $p=2$, then $G$ contains two vertices of degree $\ell$, say
$v_1,v_2$. Set $G_1=G\setminus \{v_1,v_2\}$. Since $v_1$ and $v_2$
are nonadjacent, $e(G_1)=e(G)-2\ell={{n-2}\choose{2}}$ and hence
$G_1$ is a complete graph of order $n-2$, which implies that $G\in
\mathcal {H}_n$.
\end{proof}

The following corollary is immediate from Theorem \ref{th4}.

\begin{corollary}\label{cor3}
For $1\leq \ell\leq \lfloor\frac{n+1}{2}\rfloor$ and $n\geq 12$,
$$
f(n;\overline{\kappa}_{n-1}\leq \ell)=\left\{
\begin{array}{ll}
{{n-2}\choose{2}}+2\ell, &if~1\leq \ell\leq
\lfloor\frac{n-5}{2}\rfloor,~or~\ell=
\lfloor\frac{n-3}{2}\rfloor~and~$n$~is~even;\\
{{n-2}\choose{2}}+2\ell+1,&if~\ell=
\lfloor\frac{n-3}{2}\rfloor~and~$n$~is~odd;\\
{{n-1}\choose{2}}+\ell,&if~\ell=
\lfloor\frac{n-1}{2}\rfloor~and~$n$~is~even;\\
{{n-1}\choose{2}}+2\ell-1,&if~\ell=
\lfloor\frac{n-1}{2}\rfloor~and~$n$~is~odd;\\
{{n}\choose{2}},&if~\ell=\lfloor\frac{n+1}{2}\rfloor.
\end{array}
\right.
$$
\end{corollary}

Now we focus on the edge case.

\begin{theorem}\label{th5}
Let $G$ be a connected graph of order $n\ (n\geq 12)$. If
$\overline{\lambda}_{n-1}(G)\leq \ell \ (1\leq \ell\leq
\lfloor\frac{n+1}{2}\rfloor)$, then
$$
e(G)\leq \left\{
\begin{array}{ll}
{{n-2}\choose{2}}+2\ell, &if~1\leq \ell\leq
\lfloor\frac{n-5}{2}\rfloor;\\
{{n-2}\choose{2}}+n-2,&if~\ell=
\lfloor\frac{n-3}{2}\rfloor~and~$n$~is~odd;\\
{{n-2}\choose{2}}+n-4,&if~\ell=
\lfloor\frac{n-3}{2}\rfloor~and~$n$~is~even;\\
{{n-1}\choose{2}}+n-2,&if~\ell=
\lfloor\frac{n-1}{2}\rfloor~and~$n$~is~odd;\\
{{n-1}\choose{2}}+\frac{n-2}{2},&if~\ell=
\lfloor\frac{n-1}{2}\rfloor~and~$n$~is~even;\\
{{n}\choose{2}},&if~\ell=\lfloor\frac{n+1}{2}\rfloor.
\end{array}
\right.
$$
with equality if and only if $G\in \mathcal{H}_n$ for $1\leq
\ell\leq \lfloor\frac{n-5}{2}\rfloor$; $G=K_n\setminus M$ where
$|M|=n-1$ for $\ell=\lfloor\frac{n-3}{2}\rfloor$ and $n$ odd; $G\in
\mathcal{H}_n$ for $\ell=\lfloor\frac{n-3}{2}\rfloor$ and $n$ even;
$G=K_n\setminus e$ where $e\in E(K_n)$ for
$\ell=\lfloor\frac{n-1}{2}\rfloor$ and $n$ odd; $G=K_n\setminus M$
where $|M|=\frac{n}{2}$ for $\ell=\lfloor\frac{n-1}{2}\rfloor$ and
$n$ even; $G=K_n$ for $\ell=\lfloor\frac{n+1}{2}\rfloor$.
\end{theorem}
\begin{proof}
Since $\overline{\lambda}_{n-1}(G)\leq \ell  \ (1\leq \ell\leq
\lfloor\frac{n-5}{2}\rfloor)$, it follows that
$\overline{\kappa}_{n-1}(G)\leq \overline{\lambda}_{n-1}(G)\leq
\ell$ and hence $e(G)\leq {{n-2}\choose{2}}+2\ell$ by Theorem
\ref{th4}. Suppose $e(G)={{n-2}\choose{2}}+2\ell$. Since
$\overline{\kappa}_{n-1}(G)\leq \overline{\lambda}_{n-1}(G)\leq
\ell$, we have $G\in \mathcal {H}_n$ by Theorem \ref{th4}. For
$\ell=\lfloor\frac{n+1}{2}\rfloor$, $\lfloor\frac{n-1}{2}\rfloor$
and $\lfloor\frac{n-3}{2}\rfloor$, respectively, the conclusion
holds by Propositions \ref{pro2} and \ref{pro4}.

\end{proof}

\begin{corollary}\label{cor4}
For $1\leq \ell\leq \lfloor\frac{n+1}{2}\rfloor$ and $n\geq 12$,
$$
g(n;\overline{\lambda}_{n-1}\leq \ell)=\left\{
\begin{array}{ll}
{{n-2}\choose{2}}+2\ell, &if~1\leq \ell\leq
\lfloor\frac{n-5}{2}\rfloor,~or~\ell=
\lfloor\frac{n-3}{2}\rfloor~and~$n$~is~even;\\
{{n-2}\choose{2}}+2\ell+1,&if~\ell=
\lfloor\frac{n-3}{2}\rfloor~and~$n$~is~odd;\\
{{n-1}\choose{2}}+\ell,&if~\ell=
\lfloor\frac{n-1}{2}\rfloor~and~$n$~is~even;\\
{{n-1}\choose{2}}+2\ell-1,&if~\ell=
\lfloor\frac{n-1}{2}\rfloor~and~$n$~is~odd;\\
{{n}\choose{2}},&if~\ell=\lfloor\frac{n+1}{2}\rfloor.
\end{array}
\right.
$$
\end{corollary}

\noindent{\textbf{\itshape Remark.}} It is not easy to determine the
exact value of $f(n;\overline{\kappa}_k\leq \ell)$ and
$g(n;\overline{\lambda}_k\leq \ell)$ for a general $k$. So we hope
to give a sharp lower bound of them. We construct a graph $G$ of
order $n$ as follows: Choose a complete graph $K_{k-1} \ (1\leq
\ell\leq \lfloor\frac{k-1}{2}\rfloor)$. For the remaining $n-k+1$
vertices, we join each of them to any $\ell$ vertices of $K_{k-1}$.
Clearly, $\overline{\kappa}_{n-1}(G)\leq
\overline{\lambda}_{n-1}(G)\leq \ell$ and $e(G)=
{{k-1}\choose{2}}+(n-k+1)\ell$. So $f(n;\overline{\kappa}_k\leq
\ell)\geq {{k-1}\choose{2}}+(n-k+1)\ell$ and
$g(n;\overline{\lambda}_k\leq \ell)\geq
{{k-1}\choose{2}}+(n-k+1)\ell$. From Theorems \ref{th4} and
\ref{th5}, we know that these two bounds are sharp for $k=n,n-1$.

\end{document}